\documentclass[onefignum,onetabnum]{siamart190516}

\usepackage[sort, numbers]{natbib}
\usepackage{amsmath,amsfonts,amssymb} 
\usepackage{times, fullpage}
\usepackage{mathtools,url,tikz}
\usepackage{todonotes,enumerate,etoolbox,pgfplots,algorithmic}

\newcommand{\norm}[1]{{\left\lVert#1\right\rVert}}

\newtheorem{example}[theorem]{Example}

\newcommand{\inner}[2]{{#1^{\tsp} #2}}
\def\tsp{{\sf T}}

\newcommand{\normsq}[1]{{\left\lVert#1\right\rVert^2}}

\pgfplotsset{compat=1.13}
\pgfplotsset{plotOptions/.style={%
		width=\linewidth,
		xmin=-.05,	xmax=.3,
		ymin=-.05,  ymax=1.05,
		xlabel={$\| x_0-x_* \|_{x_*}$},
		ylabel={Upper bound on $\| x_1 - x_* \|_{x_*}$},
		label style={font=\small},
		legend style={font=\small},
		legend pos=north west,
		xtick={0,.05,.1, .15,.2, .25},
		tick label style={
			font=\footnotesize,
			/pgf/number format/.cd,
			fixed,
			/tikz/.cd},
		solid,
		very thick
	}}

\title{Worst-case convergence analysis of inexact gradient and  Newton methods through semidefinite programming performance estimation\thanks{\funding{Fran\c{c}ois Glineur is
 supported by the Belgian Interuniversity Attraction Poles, and by the ARC grant 13/18-054 (Communaut\'e fran\c{c}aise de Belgique).
 Adrien Taylor  acknowledges support from the European Research Council (grant SEQUOIA 724063).}}}

\author{Etienne de Klerk\thanks{Tilburg University, The Netherlands,  \email{E.deKlerk@uvt.nl}.} \and Fran\c{c}ois Glineur\thanks{
UCL / CORE and ICTEAM, Louvain-la-Neuve, Belgium, \email{Francois.Glineur@uclouvain.be}.}
\and  Adrien B. Taylor\thanks{INRIA, D\'epartement d'informatique de l'ENS, Ecole normale sup\'erieure, CNRS, PSL Research University, Paris, France, \email{Adrien.Taylor@inria.fr}. }}

\begin{document}
\maketitle

\begin{abstract}
We provide new tools for worst-case performance analysis of the
 gradient (or steepest descent) method of Cauchy for smooth strongly convex
  functions, and Newton's method for self-concordant functions, including the
case of inexact search directions.
The analysis uses semidefinite programming performance estimation, as pioneered by Drori and Teboulle [{\em Mathematical Programming}, 145(1-2):451--482, 2014],
and extends recent performance estimation results for the method of Cauchy by the authors [{\em Optimization Letters}, 11(7), 1185-–1199, 2017].
To illustrate the applicability of the tools,
we demonstrate a novel complexity analysis of short step interior point methods using inexact search directions.
As an example in this framework, we sketch how to give a rigorous worst-case complexity analysis of a recent interior point method
by Abernethy and Hazan [\emph{PMLR}, 48:2520--2528, 2016].
\end{abstract}



\section{Introduction}
We consider the worst-case convergence of the gradient and Newton methods (with or without exact linesearch, and with possibly inexact search directions)
for certain smooth and strongly convex functions.

Our analysis is computer-assisted\footnote{Our analysis is computer-assisted in the following sense: computation was used to identify proofs, that
   could subsequently be verified in a standard mathematically rigorous way. Hence our results do not rely on the outcome of computations.} and relies on
semidefinite programming (SDP) performance estimation problems, as introduced by Drori and Teboulle \cite{drori2014}.
As a result, we develop a set of tools that may be used to design or analyse a wide range of interior point (and other) algorithms. Our analysis
is in fact an extension of the worst-case analysis of the gradient method in \cite{Klerk_Glineur_Taylor_2016}, combined with the
fact that Newton's method may be viewed as a gradient method with respect to a suitable local (intrinsic) inner product.
As a result, we obtain worst-case convergence results for a single
 iteration of a wide range of methods, including Newton's method and the steepest descent method.

\subsection*{Related work}
Our work  is similar in spirit to
recent analysis by Li et al \cite{Li_et_al_2016} of inexact proximal Newton methods for self-concordant
functions, but our approach and results are different.
 Their approach is oriented toward  inexactness from the difficulty of computing the proximal Newton step,
  whereas ours is oriented toward the difficulty of computing the Hessian.
Also, the authors of \cite{Li_et_al_2016}
do not use SDP performance estimation.

Since the seminal work by Drori and Teboulle~\cite{drori2014}, several authors have
 extended the SDP performance estimation framework. The authors of \cite{Taylor_Glineur_Hendrickx_2017a} introduced tightness guarantees
  for smooth (strongly) convex optimization, and for larger classes of
  problems in~\cite{Taylor_Glineur_Hendrickx_2017b} (where a list of sufficient conditions for applying the methodology is provided).
   It was also used to deal with nonsmooth problems~\cite{drori2016,Taylor_Glineur_Hendrickx_2017b},
     monotone inclusions and variational inequalities~\cite{Ryu2018,gu2019a,gu2019b,kim2019accelerated},
     and even to study fixed-point iterations of non-expansive operators~\cite{lieder2017convergence}.
      Fixed-step gradient descent was among the first algorithms to be studied with this methodology in different settings:
      for (possibly composite) smooth (possibly strongly)
      convex optimization~\cite{drori2014,drori2014contributions,Taylor_Glineur_Hendrickx_2017a,Taylor_Glineur_Hendrickx_2017b}, and
      its line-search version was studied using the same methodology in~\cite{Klerk_Glineur_Taylor_2016}.
      The performance estimation framework was also used for obtaining
        new methods with optimized worst-case performance guarantees in different
         settings~\cite{kim2016,drori2016,kim2018optimizing,Drori_Taylor}.
         In particular, such new methods were obtained by optimization of their algorithmic
          parameters in~\cite{kim2016,drori2016,kim2018optimizing}, and by
          analogy with conjugate-gradient type methods (doing greedy span-searches) in~\cite{Drori_Taylor}.
          Performance estimation is also related to the line of work on \emph{integral quadratic constraints} started by
           Lessard, Recht, and Packard~\cite{Lessard_Recht_Packard_2016}, which also allowed designing optimized
           methods, see~\cite{van2018fastest,Cyrus_Hu_Scoy_Lessard_2018}. The approach in~\cite{Lessard_Recht_Packard_2016}
           may be seen as a (relaxed) version of an SDP performance estimation problem where Lyapunov functions are used to
            certify error bounds~\cite{Taylor_Scoy_Lessard}.

\subsubsection*{Outline and contributions of this paper}
The contribution of this work is two-fold:
\begin{enumerate}
\item
We extend the SDP performance estimation framework to include smooth, strongly convex functions that are not defined over the entire $\mathbb{R}^n$.
By considering arbitrary inner products, we unify the analysis for gradient descent-type methods and Newton's method, thus extending the results in
\cite{Klerk_Glineur_Taylor_2016} in a significant way. In particular, we are able to give a new error analysis of inexact Newton methods for
 self-concordant functions. Thus we provide the first extension of  SDP performance estimation to  second order methods.
\item
As an {application} of the tools we develop,
we give an analysis of inexact short step interior point methods.
As an example of how our analysis may be used, we sketch how one may give a rigorous analysis of a recent interior point method by
Abernethy and Hazan \cite{Abernethy_Hazan_2016},
 where
the search direction is approximated through sampling. This particular method has sparked recent interest, since {it
links} simulated annealing and interior point methods. However, {no detailed analysis of the method is provided} in \cite{Abernethy_Hazan_2016},
 and we supply some crucial
details that are missing there.
\end{enumerate}

In Section \ref{sec:Preliminaries}, we review some basics on gradient descent methods, coordinate-free calculus, and convex functions.
Thereafter, in Section \ref{sec:Classes of convex  functions}, we describe inequalities for various classes of convex functions that
will be used in performance estimation problems. A novel aspect here is that we allow the convex domain to be arbitrary in the study of smooth, strongly convex functions.
Section \ref{sec:Performance estimation problems} introduces the SDP performance estimation problems for various classes of convex functions and domains,
and the error bounds from the analytical solutions of the performance estimation are given in Section \ref{sec:Error bounds from performance estimation}.
In Section \ref{sec:Implications for Newton} we show that our general framework for smooth, strongly convex functions includes self-concordant functions, and we are thus able to
study inexact Newton(-type) methods for self-concordant functions.
We then switch to an application of the tools we have developed, namely the complexity analysis of short step interior  point methods
that use inexact search directions (Section \ref{sec:inexactIPM}). As an example of inexact interior point methods, we consider a recent
algorithm
 by
Abernethy and Hazan \cite{Abernethy_Hazan_2016}, that uses inexact Newton-type directions for the so-called entropic
(self-concordant) barrier; see Section \ref{sec:Analysis of the method of Abernathy-Hazan}. We stress that this  example only serves to illustrate our results,
 and we do not give the full details of the analysis, which is beyond the scope of our paper.
 The additional analysis on how sampling may be used in the method by
Abernethy and Hazan \cite{Abernethy_Hazan_2016} to compute the inexact Newton-type directions is given in the separate work \cite{Badenbroek_DeKlerk2018}.


\section{Preliminaries}
\label{sec:Preliminaries}

Throughout $f$ denotes a differentiable convex function, {whose domain is denoted by $D_f \subset \mathbb{R}^n$, whereas $D$ is used to denote open sets that may not correspond to the domain of $f$}.
We will indicate additional assumptions on  $f$ as needed. We will mostly use the notation from the book by Renegar \cite{Renegar2001a}, for easy reference.

\subsection{Gradients and Hessians}
In what follows we fix a reference inner product $\langle \cdot,\cdot\rangle$ on $\mathbb{R}^n$ with induced norm $\|\cdot\|$.

\begin{definition}[Gradient and Hessian]
\label{def:gradient}
If $f$ is  differentiable, the gradient of $f$ at $x \in D_f$ with respect to $\langle \cdot , \cdot \rangle$
is the unique vector $g(x)$ such that
\[
\lim_{\| \Delta x\| \rightarrow 0} \frac{f(x+\Delta x) - f(x) - \langle g(x) , \Delta x \rangle }{\|\Delta x\|} = 0.
\]
If $f$ is twice differentiable, the second derivative (or Hessian) of $f$ at $x$ is defined as the (unique) linear
operator $H(x)$ that satisfies
\[
\lim_{\|\Delta x\| \rightarrow 0} \frac{\|g(x+\Delta x) - g(x)-H(x)\Delta x\|}{\|\Delta x\|} = 0.
\]
\end{definition}

Note that $g(x)$, and therefore also $H(x)$, depend on the reference inner product. If $\langle \cdot,\cdot\rangle$ is the Euclidean dot product then
 $g({x}) = \nabla f({x}) = \left[\frac{\partial f({x})}{\partial x_i}\right]_{i=1,\ldots,n}$, and
 the Hessian, when written as a matrix with respect to the standard basis, takes the familiar form $[H(x)]_{ij} =: [\nabla^2 f(x)]_{ij} = \frac{\partial^2f(x) }{\partial x_i \partial x_j}$
 $(i,j \in \{1,\ldots,n\})$.

 If $B:\mathbb{R}^n \rightarrow \mathbb{R}^n$ is a self-adjoint positive definite linear operator, we may define a
 new inner product in terms of the reference inner product as follows:
 $\langle \cdot,\cdot\rangle_B$ via $\langle {x}, {y}\rangle_{B} = \langle {x}, B{y}\rangle$ $\forall x,y \in \mathbb{R}^n$.
 (Recall that  all inner products in $\mathbb{R}^n$ arise in this way.)
If we change the inner product in this way, then the gradient changes to $B^{-1} g({x})$, and the Hessian at $x$ changes to  $B^{-1}H(x)$.

Recall that $H(x)$ is self-adjoint with respect to the reference inner product if $f$ is twice continuously differentiable.
Assuming that $H(x)$ is positive definite and self-adjoint at a given point $x$, define the intrinsic (w.r.t. $f$ at $x$) inner product
\[
\langle u,v\rangle_{x} :=  \langle u,v\rangle_{H(x)} \equiv \langle
u,H(x)v\rangle.
\]
 The definition is {\em independent of the reference inner product}
$\langle \cdot, \cdot \rangle$.
The induced norm for the intrinsic inner product is denoted by: $\|u\|_x = \sqrt{\langle u,u\rangle_{x}}$.
For the intrinsic inner product,
the gradient at $y$ is denoted by $g_x(y) := H(x)^{-1}g(y)$, and
the Hessian at $y$ by  $H_x(y):=H(x)^{-1}H(y)$.

\subsection{Fundamental theorem of calculus}
In what follows, we will recall coordinate-free versions of the fundamental theorem
of calculus. Our review follows Renegar \cite{Renegar2001a}, and all proofs may be found there.
\begin{theorem}[Theorem 1.5.1 in \cite{Renegar2001a}]
\label{thm:151}
If $x,y \in D_f$, then
\[
f(y) - f(x) = \int_0^1  \langle g(x + t(y-x)), y-x\rangle dt.
\]
\end{theorem}

Next, we recall the definition of a vector-valued integral.

\begin{definition}
Let $t \mapsto v(t) \in \mathbb{R}^n$ where $t \in [a,b]$.
Then $u$ is the integral of $v$ if
\[
\langle u,w\rangle = \int_a^b \langle v(t),w\rangle \; dt \mbox{ for all } w \in \mathbb{R}^n.
\]
\end{definition}
Note that this definition is in fact independent of the reference inner product.
We will use the following bound on norms of vector-valued integrals.

\begin{theorem}[Proposition 1.5.4 in \cite{Renegar2001a}]
\label{prop:154}
Let $t \mapsto v(t) \in \mathbb{R}^n$ where $t \in [a,b]$.
If $v$ is integrable, then
\[
\left\|\int_a^b v(t)dt \right\| \le \int_a^b \|v(t)\|dt.
\]
\end{theorem}

Finally, we will  require the following version of the fundamental theorem for gradients.

\begin{theorem}[Theorem 1.5.6 in \cite{Renegar2001a}]
\label{thm:156}
If $x,y \in D_f$, then
\[
g(y)-g(x) = \int_0^1 H(x+t(y-x))(y-x)dt.
\]
\end{theorem}

\subsection{Inexact gradient and Newton methods}

We consider approximate gradients
${d}_i \approx  g(x_i)$ at a given iterate $x_i$ $(i=0,1,\ldots)$; to be precise we assume the following for a given $\varepsilon \in [0,1)$:
\begin{equation}
\label{eq:error}
\|d_i-g({x}_i)\| \le \varepsilon\|g({x}_i)  \| \quad i = 0,1,\ldots
\end{equation}
Note that $\varepsilon = 0$ yields the gradient, i.e.\ $d_i = g(x_i)$.

\begin{algorithm}
\caption{Inexact gradient descent method}
\label{alg:buildtree}
\begin{algorithmic}
\STATE{\textbf{Input:} $f:\mathbb{R}^n\rightarrow \mathbb{R}$, ${x}_0\in \mathbb{R}^n$, ${0\le \varepsilon < 1}$}
\STATE{\textbf{for} $i=0,1,\ldots$}
\STATE{Select any  ${d}_i$ that satisfies \eqref{eq:error} }
\STATE{Choose a step size $\gamma > 0$}
\STATE{Update ${x}_{i+1}={x}_i-\gamma {d}_i$}
\end{algorithmic}
\end{algorithm}

We will consider two ways of choosing the step size $\gamma$:
\begin{enumerate}
\item
Exact line search:
$\gamma=\text{argmin}_{\gamma \in \mathbb{R}}f\left({x}_i-\gamma {d}_i\right)$;
\item
Fixed step size: $\gamma$ takes the same value at each iteration, and this value is known beforehand.
\end{enumerate}
We note once more that, at iteration $i$, we obtain the Newton direction by using the $\langle \cdot,\cdot \rangle_{x_i}$ inner product (if $\varepsilon = 0$).
More generally, by using an inner product $\langle \cdot,\cdot\rangle_{B}$ for some positive definite, self-adjoint operator $B$, we obtain
the direction $-B^{-1}g(x_i)$, which is the Newton direction when $B = H(x_i)$, a quasi-Newton direction if $B \approx H(x_i)$, and the familiar
steepest descent direction $-\nabla f(x_i)$ if $B = I$.

\section{Classes of convex  functions}
\label{sec:Classes of convex  functions}
In this section we review two classes of convex functions, namely smooth, strongly convex functions and self-concordant functions.
We also show that, in a certain sense, the latter class may be seen as a special case of the former, by extending some known results
on the first class.

\subsection{Convex functions}
Recall that a differentiable function $f$
is convex on  open convex set ${D} \subset \mathbb{R}^n$ if and only if
\begin{equation}
\label{convex_f}
f(y) \ge  f(x) + \langle g(x),y-x\rangle \quad \forall x,y \in {D}.
\end{equation}
Also recall that a twice continuously differentiable function is convex on ${D}$ if and only if $H(x) \succeq 0$ for all $x \in {D}$.

\subsection{Smooth strongly convex functions}
A differentiable function $f$ with $D_f = \mathbb{R}^n$ is called $L$-smooth and $\mu$-strongly convex if it satisfies the following two properties:
\begin{itemize}
\item[(a)] {\bf $L$-smoothness}: there exists some $L>0$ such that $\frac{1}{L}\norm{ g(u) - g(v)}\leq \norm{u-v}$
holds for all pairs  $u,v$ and corresponding gradients $g(u), g(v)$.
\item[(b)] {\bf $\mu$-strong convexity}: there exists some $\mu>0$ such that the function $x \mapsto f(x)-\frac{\mu}{2}\|{x}\|^2$ is convex.
\end{itemize}
The class of such functions is denoted by $\mathcal{F}_{\mu,L}(\mathbb{R}^n)$.
Note that this function class is defined in terms of the reference inner product and its induced norm.
In particular, it is not invariant under a change of inner product: under a change of inner product a smooth, strongly convex function
remains smooth, strongly convex, but the parameters $\mu$ and $L$ depend on the inner product.
For our purposes, it is important not to fix the inner product a priori.

If $D\subseteq \mathbb{R}^n$ is an open, convex set, then we denote the set of functions that satisfy properties (a) and (b) on $D$ by $\mathcal{F}_{\mu,L}(D)$.
The following lemma refines and extends well-known results from the literature on smooth convex functions. The novelty here is that we
allow for arbitrary inner products, and any open, convex domain $D$.
Similar results, are given, for example in the textbook by Nesterov \cite[Chapter 2]{nesterov_lectures_convex_opt},
and we only prove the results not covered there. We will use the L\"owner partial order notation: for symmetric matrices (or self-adjoint linear operators)
$A$ and $B$, `$A \preceq B$' means that $B-A$ is positive semidefinite. We will denote the identity matrix (or operator) by $I$.

\begin{lemma} Let $D$ be an open convex set and $f:D\rightarrow \mathbb{R}$ be twice continuously differentiable. The following statements are equivalent:
\label{lemm:Lsmooth}
\begin{itemize}
\item[(a)] $f$ is convex and $L$-smooth on $D$,
\item[(b)] $0\preceq H(x) \preceq L I$ $\;\forall x\in D$,
\item[(c)] $\langle{g(y)-g(x)},{y-x}\rangle \geq \frac{1}{L}\normsq{g(y)-g(x)}$ $\;\forall x,y\in D$.
\end{itemize}
\end{lemma}
\begin{proof}
$(a)\Rightarrow(b)$:
The proof of this implication is similar to that
of \cite[Theorems 2.1.5 and 2.1.6]{nesterov_lectures_convex_opt} (see relation (2.1.16) there), and is therefore omitted here.
%

 $(b)\Rightarrow(c)$:
If $\|H(x)\| \le L$, then $H(x) - \frac{1}{L}H^2(x) \succeq 0$ for all $x \in D$, and Theorem \ref{thm:156} implies
 \begin{eqnarray*}
\langle g(y)-g(x),y-x\rangle &=& \left\langle y-x,\int_0^1 H(x+t(y-x))(y-x)dt\right\rangle \\
&=& \int_0^1\langle y-x, H(x+t(y-x))(y-x)\rangle dt \\
&\ge & \int_0^1\langle y-x, \frac{1}{L}H^2(x+t(y-x))(y-x)\rangle dt \\
&=& \frac{1}{L} \int_0^1 \| H(x+t(y-x))(y-x)\|^2 dt \\
&\ge & \frac{1}{L} \left(\int_0^1 \| H(x+t(y-x))(y-x)\| dt\right)^2 \quad \mbox{(Jensen inequality)} \\
& \ge & \frac{1}{L} \left\|\int_0^1  H(x+t(y-x))(y-x) dt\right\|^2 \quad \mbox{(Theorem \ref{prop:154})}\\
&=& \frac{1}{L}\|g(y)-g(x)\|^2 \quad \mbox{(Theorem \ref{thm:156})}.
\end{eqnarray*}
  $(c)\Rightarrow(a)$:
  Condition $(c)$, together with the Cauchy-Schwartz inequality, immediately imply $L$-smoothness.
  To show convexity, note that, by Theorem \ref{thm:151},
  \begin{eqnarray*}
  f(y) - f(x) -\langle g(x),y-x\rangle &=& \int_0^1  \frac{1}{t}\langle g(x + t(y-x))-g(x), t(y-x)\rangle dt \\
                                           &\ge&  \int_0^1  \frac{1}{tL}\| g(x + t(y-x))-g(x)\|^2 dt \ge 0, \\
    \end{eqnarray*}
  where the first inequality is from condition $(c)$. Thus we obtain the convexity inequality \eqref{convex_f}.
\end{proof}

An interesting question  is to understand the class of functions where the following inequality
 holds
\begin{equation}\label{FmuL_Rn}
  f(y)-f(x) - \langle{g(x)},{y-x}\rangle \geq \frac{1}{2L}\normsq{g(y)-g(x)}
\end{equation}
for all $x,y$ in a given open convex set $D$, where $L>0$ is fixed. (Note that \eqref{FmuL_Rn} implies
condition (c) in the lemma, by adding  \eqref{FmuL_Rn} to itself after interchanging $x$ and $y$.)

Indeed, for performance estimation problems, one attempts to find a  function from a
specific class that corresponds to the worst-case input for a given iterative algorithm.
It is therefore of both practical and theoretical interest to understand the function class that satisfies \eqref{FmuL_Rn}.

The inequality \eqref{FmuL_Rn} is known to hold for $D = \mathbb{R}^n$ if, and only if, $f \in \mathcal{F}_{0,L}(\mathbb{R}^n)$,
by results of Taylor, Glineur and Hendrickx \cite{Taylor_Glineur_Hendrickx_2017a} (see also Azagra and  Mudarra \cite{Azagra_Mudara_2017}).
In an earlier version of this paper, we asked if this is true for more general open convex sets $D \subset \mathbb{R}^n$,
but this turns out not to be the case: Drori \cite{drori2018}  recently constructed an example of a bivariate  function
with open domain, say $D$, such that  $f \in \mathcal{F}_{0,L}(D)$ with $L =1$, but where \eqref{FmuL_Rn} does not hold for all $x,y \in D$.
For this reason, item (c) in Lemma  \ref{lemm:Lsmooth}, cannot be changed to the stronger condition \eqref{FmuL_Rn}.

 Lemma \ref{lemm:Lsmooth} allows us to derive  the following necessary and sufficient conditions for membership of $\mathcal{F}_{\mu,L}(D)$.
The condition $(d)$ below is new, and will be used extensively in the proofs that follow.
\begin{theorem}
\label{thm:mu_L_H}
Let $D$ be an open convex set and $f:D\rightarrow \mathbb{R}$ be twice continuously differentiable. The following statements are equivalent:
\begin{itemize}
\item[(a)] $f$ is $\mu$-strongly convex and $L$-smooth on $D$, i.e.\ $f \in \mathcal{F}_{\mu,L}(D)$,
\item[(b)] $\mu I\preceq H(x) \preceq L I$ $\forall x\in D$,
\item[(c)] $f(x)-\frac{\mu}{2}\normsq{x}$ is convex and $(L-\mu)$-smooth on $D$,
\item[(d)] for all $x,y\in D$ we have
 \begin{equation}
 \label{gradient inequality}
 \langle{g(x)-g(y)},{x-y}\rangle\geq \frac{1}{1-\frac{\mu}{L}} \left(\frac{1}{L}\normsq{g(x)-g(y)}+\mu\normsq{x-y}
 -2\frac{\mu}{L}\langle{g(x)-g(y)},{x-y}\rangle\right).
 \end{equation}
\end{itemize}
\end{theorem}
\begin{proof}
The equivalences $(a)\Leftrightarrow (b)\Leftrightarrow(c)$ follow directly from Lemma \ref{lemm:Lsmooth} and the relevant definitions.

 $(c)\Leftrightarrow(d)$: Requiring $h(x)=f(x)-\frac{\mu}{2}\normsq{x}$ to be convex and $(L-\mu)$-smooth on $D$ can equivalently be formulated as requiring $h$ to satisfy
\[\langle{g_h(x)-g_h(y)},{x-y}\rangle \geq \frac{1}{L-\mu}\normsq{g_h(y)-g_h(x)}\] for all $x,y\in D$, where $g_h$ is the gradient of $h$. Equivalently:
\begin{align*}
(L-\mu)\left[\langle{g_f(x)-g_f(y)},{x-y}\rangle-\mu \normsq{x-y}\right]\geq &\normsq{g_f(y)-g_f(x)}+\mu^2\normsq{x-y}-2\mu\langle{g_f(y)-g_f(x)},{y-x}\rangle,
\end{align*}
which is exactly condition $(d)$ in the statement of the theorem.
\end{proof}

We note once more that the spectrum of the Hessian is not invariant under change in inner product, thus the values $\mu$ and $L$ are
intrinsically linked to the reference inner product. If $D = \mathbb{R}^n$, a stronger condition than condition $(d)$ in the last theorem holds, namely
\begin{equation}
\label{eq:FmuL}
f(y)-f(x) - \langle{g(x)},{y-x}\rangle \geq \frac{1}{2(1-\frac{\mu}{L})} \left(\frac{1}{L}\normsq{g(y)-g(x)}+\mu\normsq{x-y}-2\frac{\mu}{L}\langle{g(x)-g(y)},{x-y}\rangle\right).
\end{equation}
For the reasons discussed after Lemma \ref{lemm:Lsmooth}, the inequality \eqref{eq:FmuL} does not hold in general if $D \neq \mathbb{R}^n$, due to the results by Drori \cite{drori2018}.

\subsection{Self-concordance}
Self-concordant functions are special convex functions
introduced by Nesterov and Nemirovski \cite{int:Nesterov5}, that play a certral role in the analysis of interior point algorithms.
We will use the (slightly) more general definition by Renegar \cite{Renegar2001a}.

\begin{definition}[Self-concordant functional]
Let $f:D_f \rightarrow \mathbb{R}$ be such that $H(x) \succ 0$ for all $x \in D_f$, and let $B_x(x,1)$ denote the open unit ball centered at $x \in D_f$ for the $\|\cdot\|_x$ norm. Then
$f$ is called self-concordant  if:
\begin{enumerate}
\item
For all $x \in D_f$ one has $B_x(x,1) \subseteq D_f$;
\item
For all $y \in B_x(x,1)$ one has
\[
1 - \|y-x\|_x \le \frac{\|v\|_y}{\|v\|_x} \le \frac{1}{1 - \|y-x\|_x} \mbox{  for all } v \neq 0.
\]
\end{enumerate}
\end{definition}

An equivalent characterization of self-concordance, due to Renegar \cite{Renegar2001a}, is as follows.

\begin{theorem}[Theorem 2.2.1 in \cite{Renegar2001a}]
\label{thm:alt def sc}
Assume $f$ such that for all $x \in D_f$ one has $B_x(x,1) \subseteq D_f$.

Then $f$ is self-concordant if, and only if, for all $x \in D_f$ and $y \in B_x(x,1)$:
\begin{equation}
\label{def:altSC}
\max\left\{\|H_x(y)\|_x, \|H_x(y)^{-1}\|_x\right\} \le \frac{1}{(1 - \|y-x\|_x)^2}.
\end{equation}
\end{theorem}

This alternative characterization allows us to establish a new link between
self-concordant and $L$-smooth, $\mu$-strongly convex functions.

\begin{corollary}
\label{thm:sc_vs_FmuL}
Assume $f:D_f \rightarrow \mathbb{R}$ is self-concordant, $x \in D_f$, and $\delta < 1$.
Then $f \in F_{\mu,L}(D)$ with respect to the inner product $\langle \cdot,\cdot\rangle_x$, if
$
D = \left\{y \; | \; \|y-x\|_x <\delta  \right\} = B_x(x,\delta),
$
$
\mu = (1-\delta)^2$, and $L = \frac{1}{(1-\delta)^2}$.

Conversely, assume $f:D_f \rightarrow \mathbb{R}$ is twice continuously differentiable,
and $H(x) \succ 0$ for all $x \in D_f$. If, for all $x \in D_f$ and $\delta \in (0,1)$, it holds that
\begin{enumerate}
  \item
  $D := B_x(x,\delta) \subset D_f$,
  \item
   $f \in F_{\mu,L}(D)$ with respect to the inner product $\langle \cdot,\cdot\rangle_x$, and
$
\mu = (1-\delta)^2$ and $L = \frac{1}{(1-\delta)^2}$,
\end{enumerate}
then $f$ is self-concordant.
\end{corollary}
\begin{proof}
For the first implication, observe that for a self-concordant $f$,
Theorem \ref{thm:alt def sc} implies the spectrum of $H_x(y)$ is contained in the interval
$\left[(1 - \|y-x\|_x)^2,\frac{1}{(1 - \|y-x\|_x)^2}\right]$, which in turn is contained
in  $\left[(1 - \delta)^2,\frac{1}{(1 - \delta)^2}\right]$ for all $y \in B_x(x,\delta)$.
Theorem \ref{thm:mu_L_H} now yields the required result.

To prove the converse, assume $y \in B_x(x,1)$ and set $\delta = \|x-y\|_x$, $D = B_x(x,\delta)$,
$\mu = (1-\delta)^2$, and $L = \frac{1}{(1-\delta)^2}$.

 Since  $f \in F_{\mu,L}(D)$ by assumption,
it holds that
\[
\mu I \preceq H_x(y) \preceq L I,
\]
which is the same as the condition \eqref{def:altSC} that guarantees self-concordance.
\end{proof}

\section{Performance estimation problems}
\label{sec:Performance estimation problems}
Performance estimation problems, as introduced by Drori and Teboulle \cite{drori2014}, are semidefinite programming (SDP) problems that
bound the worst-case performance of certain iterative optimization algorithms. Essentially, the goal is to find the
objective function from a given function class, that exhibits the worst-case behavior  for a given iterative algorithm.

In what follows, we list the SDP performance estimation problems that we will employ to study worst-case bounds for one iteration of gradient methods,
applied to a smooth, strongly convex $f$ that admits a minimizer, denoted by $x_*$.
These performance estimation problems have variables that correspond to (unknown) iterates $x_0$ and $x_1$, the minimizer $x_*$, as well
as the gradients and function values at these points, namely $g_i$   correspond to $g(x_i)$ ($i \in \{*,0,1\})$), and
$f_i$  correspond to $f(x_i)$ ($i \in \{*,0,1\})$). We may assume $x_* = g_* = 0$ and $f_* = 0$ without loss of generality.

The objective is
to identify the worst-case after one iteration is performed, namely to find the maximum value of either
$f_1 - f_*$, $\|g_1\|$, or $\|x_1-x_*\|$,
given an upper bound on the initial value of one of the quantities $f_0 - f_*$, $\|g_0\|$, or $\|x_0-x_*\|$ (this upper bound will be denoted $R$ below).

Note again that the norm may be any induced norm on $\mathbb{R}^n$.

Note that we only consider one iteration, i.e.\ the transition from $x_0$ to $x_1$, whereas
SDP performance estimation is typically used to study several iterations; see e.g.\ \cite{drori2014,Taylor_Glineur_Hendrickx_2017a}.
For our purposes, it suffices to consider one iteration --- we will elaborate on this later on. For the time being,
we consider functions where the domain is all of $\mathbb{R}^n$, and will move to restricted domains later on.

\subsection*{Performance estimation with exact line search}
\begin{itemize}
\item
Parameters: $L \ge \mu > 0$, $R > 0$;
\item
Variables:
$\left\{({x}_i,{g}_i,f_i) \right\}_{i\in S}$ $\;\; (S = \{*,0,1\})$.
\end{itemize}

\subsubsection*{Worst-case function value}
\begin{equation}
\label{pep:els_f}
\left.
\begin{array}{rl}
\max \ &f_1-f_* \\
\mbox{s.t. } &
f_i - f_j - \langle{{g}_j},{{x}_i-{x}_j}\rangle \geq \frac{1}{2(1-\mu/L)}\left( \frac{1}{L}\norm{{g}_i-{g}_j}^2+ \mu \norm{{x}_i-{x}_j}^2 - 2\frac{\mu}{L} \langle{{g}_j-{g}_i},{{x}_j-{x}_i}\rangle\right) \; \forall i,j \in S\\
& {g}_* = {0}  \\
&\langle {x}_{1}-{x}_0, {g}_1\rangle = 0 \\
 & \langle{g}_0 , {g}_1\rangle \le \varepsilon \|g_0\|\|g_1\| \\
&f_0-f_* \leq R  \\
\end{array}
\right\}
\end{equation}

The first constraint corresponds to
\eqref{eq:FmuL}, and models the necessary condition for $f \in {\mathcal{F}}_{\mu,L}(\mathbb{R}^n)$.
The second constraint corresponds to the fact that the gradient is zero at a minimizer. The third  constraint
is the well-known property of exact line search,
while the fourth constraint is satisfied
if the approximate gradient condition \eqref{eq:error} holds.
Finally, the fifth constraint ensures that the problem is bounded.

Note that the resulting problem may be written as an SDP problem, with $4\times 4$ matrix variable given
by the Gram matrix of the vectors $x_0,x_1,g_0,g_1$ with respect to the reference inner product.
In particular the fourth contraint may be written as the linear matrix inequality:
\begin{equation}
\label{noisy lmi}
\left(
     \begin{array}{cc}
      \varepsilon \|{g}_0\|^2 & \langle{g}_0, {g}_1\rangle\\
\langle{g}_0, {g}_1\rangle & \varepsilon \|{g}_1\|^2 \\
     \end{array}
   \right)
\succeq 0.
\end{equation}

Also note that the optimal value of the resulting SDP problem is independent of the inner product.
The SDP problem \eqref{pep:els_f} was first studied in \cite{Klerk_Glineur_Taylor_2016}.

\subsubsection*{Worst-case gradient norm}
The second variant of performance estimation is to find the worst case convergence of the gradient norm.
\begin{equation}
\label{pep:els_g}
\left.
\begin{array}{rll}
\max \ &\|g_1\|^2 &\\
\mbox{s.t. } &
{\langle{g_i-g_j},{x_i-x_j}\rangle\geq \frac{1}{1-\frac{\mu}{L}}\left(\frac{1}{L}\|{g_i-g_j}\|^2+\mu\|{x_i-x_j}\|^2-2\frac{\mu}{L}\langle{g_i-g_j},{x_i-x_j}\rangle\right)}  & \forall i,j \in S\\
& {g}_* = {0}  &\\
&\langle {x}_{1}-{x}_0, {g}_1\rangle = 0 &\\
& \langle{g}_0 , {g}_1\rangle \le \varepsilon \|g_0\|\|g_1\| &\\
&\|g_0\|^2 \leq R & \\
\end{array}
\right\}
\end{equation}

\subsubsection*{Worst-case distance to optimality}
The third variant of performance estimation is to find the worst case convergence of the distance to optimality.
\begin{equation}
\label{pep:els_x}
\left.
\begin{array}{rll}
\max \ &\|x_1-x_*\|^2 &\\
\mbox{s.t. } &
{\langle{g_i-g_j},{x_i-x_j}\rangle\geq \frac{1}{1-\frac{\mu}{L}}\left(\frac{1}{L}\|{g_i-g_j}\|^2+\mu\|{x_i-x_j}\|^2-2\frac{\mu}{L}\langle{g_i-g_j},{x_i-x_j}\rangle\right)} & \forall i,j \in S\\
& {g}_* = {0}  &\\
&\langle {x}_{1}-{x}_0, {g}_1\rangle = 0 &\\
& \langle{g}_0 , {g}_1\rangle \le \varepsilon \|g_0\|\|g_1\| &\\
&\|x_0-x_\star \|^2 \leq R & \\
\end{array}
\right\}
\end{equation}
In what follows we will give upper bounds on the optimal values of these performance estimation SDP problems.

\subsection*{PEP with fixed step sizes}
For fixed step sizes, the performance estimation problems \eqref{pep:els_f}, \eqref{pep:els_g}, and \eqref{pep:els_x}
change as follows:
\begin{enumerate}
\item
for given step size $\gamma > 0$, the condition $x_1 = x_0 - \gamma d$ is used to eliminate $x_1$,
where $d$ is the approximate gradient at $x_0$.
\item
The vector $d$ is viewed as a variable in the performance estimation problem.
\item
The exact line search condition, $\langle {x}_{1}-{x}_0, {g}_1\rangle = 0$, is omitted.
\item
The condition
$
\normsq{d-g_0}\leq \varepsilon^2\normsq{g_0}
$
is added, that corresponds to \eqref{eq:error}, and $\langle{g}_0 , {g}_1\rangle \le \varepsilon \|g_0\|\|g_1\|$ is omitted.
\end{enumerate}

\section{Error bounds from performance estimation}
\label{sec:Error bounds from performance estimation}
The optimal values of the performance estimation problems in the last section give bounds on the worst-case convergence rates of the gradient method for different performance measures. In this section, we provide bounds that were obtained using the solutions to the previously presented performance estimation problems. As the bounds we present below are not always the \emph{tightest} ones, each result is provided with comments regarding tightness.


To the best of our knowledge, there are no results dealing with the exact same settings in the literature.
Among others, one can find detailed analyses of first-order method under deterministic
uncertainty models in~\cite{Daspremont2008,Devolder2014,Schmidt2011}. Related results involving relative inaccuracy
 (as in this work) can be found in  older works by Polyak \cite{Polyak1971} where  the focus is on smooth convex minimization.

To give some meaningful comparison, we will compare our results to standard ones for the simpler case $\varepsilon=0$.

\subsection{PEP with exact line search}
The first result concerns error bounds for the inexact gradient method with exact line
 search when no restriction is applied on the domain.
 This problem, originally due to Cauchy, was studied in detail in \cite{Klerk_Glineur_Taylor_2016} using SDP performance estimation.
 Here we will generalize the main result from \cite{Klerk_Glineur_Taylor_2016} to include arbitrary inner products.
 The Cauchy problem remains of enduring interest; see e.g.\ the recent work by Bolte and Pauwels \cite[\S 5.3]{Bolte2020}.

\begin{theorem}
\label{thm:ELS improvements}
Consider the inexact gradient method with exact line search applied to some $f \in \mathcal{F}_{\mu,L}(\mathbb{R}^n)$.
If $\kappa := \frac{\mu}{L}$ {(the inverse condition number: $\kappa \in (0,1]$)} and $\varepsilon \in \left[0, \frac{2\sqrt{\kappa}}{1+\kappa}\right]$, one has
\begin{eqnarray*}
{f(x_1)} - f(x_*)&\le & \left(\frac{1-\kappa +\varepsilon(1-\kappa)}{1+\kappa +\varepsilon(1-\kappa)}\right)^2 ({f(x_0)}-f(x_*)), \\
\|g(x_1)\| &\le & \left(\varepsilon+ \sqrt{1-{\varepsilon}^2}\frac{1-\kappa}{2\sqrt{\kappa}}\right)\|{g(x_0)}\|, \\
\|{x_1-x_*}\| & \le & \left(\varepsilon+ \sqrt{1-{\varepsilon}^2}\frac{1-\kappa}{2\sqrt{\kappa}}\right)\|{x_0-x_*}\|. \\
\end{eqnarray*}
\end{theorem}

\begin{proof}
The three inequalities in the statement of the theorem follow
 from corresponding upper bounds on the optimal values of the SDP performance estimation problems \eqref{pep:els_f}, \eqref{pep:els_g}, and \eqref{pep:els_x},
  respectively.

The first of the SDP performance estimation problems, namely \eqref{pep:els_f}, is exactly the same as the one in \cite{Klerk_Glineur_Taylor_2016} (see (8) there).
The first inequality therefore follows from
 Theorem 1.2 in \cite{Klerk_Glineur_Taylor_2016}.

It remains to demonstrate suitable upper bounds on the SDP performance estimation problems  \eqref{pep:els_g}, and \eqref{pep:els_x}.
This is done by aggregating the constraints of these respective SDP problems by using suitable (Lagrange)  multipliers.\footnote{The
multipliers used in the proof were obtained by solving the SDP problems
 \eqref{pep:els_g} and \eqref{pep:els_x} numerically for different values of $\mu$ and $L$ (one may assume w.l.o.g.\ that $R=1$), and subsequently
guessing the correct analytical expressions of the multipliers by looking at the optimal solution of the dual SDP problem.
 For example it quickly became clear that the second multiplier was $(L+\mu)$,
and that the matrix $S$  always had rank one. The correctness of these expressions for the multipliers is verified in the proof.
 Thus numerical computations were only used to
find the proof of Theorem \ref{thm:ELS improvements}, i.e.\ to guess the correct analytical expressions for the multipliers, and play no role in the proof itself.}

To this end, consider the following constraints from \eqref{pep:els_g} and their associated multipliers:
\begin{align*}
&{\langle{g_0-g_1},{x_0-x_1}\rangle\geq \frac{1}{1-\frac{\mu}{L}}\left(\frac{1}{L}\|{g_0-g_1}\|^2+\mu\|{x_1-x_0}\|^2-2\frac{\mu}{L}\langle{g_0-g_1},{x_0-x_1}\rangle\right)}  &&:{(L-\mu)}\lambda,\\
&\langle{g_1},{x_1-x_0}\rangle= 0 && :L+\mu,\\
&\begin{pmatrix}
\epsilon ||g_0||^2 & \langle{g_0},{g_1}\rangle\\
\langle{g_0},{g_1}\rangle & \epsilon ||g_1||^2
\end{pmatrix}\succeq 0 &&:S,
\end{align*}
with $S=\begin{pmatrix}
s_{11} & s_{12}\\ s_{12} & s_{22}
\end{pmatrix},$ and:
\begin{align*}
&\lambda=\frac{2  \varepsilon \sqrt{\kappa}}{\sqrt{1 - {\varepsilon}^2} \left(1 - \kappa\right)} + 1,
& \\
&s_{11}=\frac{3 \varepsilon}{2} - \frac{\varepsilon \left(\kappa + \frac{1}{\kappa}\right)}{4} + \frac{1-\kappa}{2 \sqrt{\kappa(1 - {\varepsilon}^2)}} - \frac{{\varepsilon}^2 \left(1 -  \kappa\right)}{\sqrt{\kappa(1 - {\varepsilon}^2)}},
&\\
&s_{22}= \frac{2  \sqrt{\kappa(1 - {\varepsilon}^2)}-\varepsilon \left(1 - \kappa\right)}{\left(1-{\varepsilon}^2\right) \left(1 - \kappa\right) + 2 \varepsilon \sqrt{\kappa(1 - {\varepsilon}^2)}}, &
\\
&s_{12}=
\frac{\varepsilon\, \left(1 - \kappa\right)}{2  \sqrt{\kappa(1 - {\varepsilon}^2)}} - 1.&
\end{align*}
{Assuming the corresponding multipliers are of appropriate signs (see discussion below), the proof consists in reformulating the following weighted sum of the previous inequalities (the validity of this inequality follows from the signs of the multipliers):
\begin{equation}\label{eq:weightedsum_grad}
\begin{aligned}
0\geq &(L-\mu)\lambda\left[\frac{1}{1-\frac{\mu}{L}}\left(\frac{1}{L}\|{g_0-g_1}\|^2+\mu\|{x_1-x_0}\|^2-2\frac{\mu}{L}\langle{g_0-g_1},{x_0-x_1}\rangle\right)-\langle{g_0-g_1},{x_0-x_1}\rangle\right]\\&+(L+\mu)\left[\langle{g_1},{x_1-x_0}\rangle\right]-\mathrm{Trace}\left(\begin{pmatrix}
s_{11} & s_{12} \\ s_{12} & s_{22}
\end{pmatrix}\begin{pmatrix}
\epsilon ||g_0||^2 & \langle{g_0},{g_1}\rangle\\
\langle{g_0},{g_1}\rangle & \epsilon ||g_1||^2
\end{pmatrix} \right).
\end{aligned}
\end{equation}
We first show that  multipliers are nonnegative (resp.\ positive semidefinite) where required; that is, $(L-\mu)\lambda\geq 0$ and $S\succeq 0$.
 The nonnegativity of the first term is clear from $0 < \mu \le L$ and $\lambda \ge 0$.
 Concerning $S$, let us note that $s_{22}\geq 0\Leftrightarrow \varepsilon\in\left[\frac{-2\sqrt{\kappa}}{1+\kappa},\frac{2\sqrt{\kappa}}{1+\kappa}\right]$.}
When $\varepsilon < \frac{2\sqrt{\kappa}}{1+\kappa}$, $s_{22}$ ensures that there exists a positive eigenvalue for $S$, since $s_{22}>0$.
In order to prove that both eigenvalues of $S$ are nonnegative, one may verify:
\begin{equation*}
\det S=s_{11}s_{22}-s_{12}^2=0.
\end{equation*}
{Therefore, one eigenvalue of $S$ is positive and the other one is zero when $\varepsilon < \frac{2\sqrt{\kappa}}{1+\kappa}$,
and in the simpler case $\varepsilon = \frac{2\sqrt{\kappa}}{1+\kappa}$, we have $S=0$, and hence the inequality~\eqref{eq:weightedsum_grad} is valid.}

{Reformulating the valid inequality~\eqref{eq:weightedsum_grad} yields:}
\begin{align*}
\|{g_1}\|^2\leq &\left(\varepsilon+ \sqrt{1-{\varepsilon}^2}\frac{1-\kappa}{2\sqrt{\kappa}}\right)^2 \|{g_0}\|^2\\
&-\kappa \frac{2 \varepsilon \sqrt{\kappa}+(1-\kappa)\sqrt{1-\varepsilon^2} }{(1-\kappa)\sqrt{1-\varepsilon^2}}
\normsq{\frac{\varepsilon (1+\kappa)}{\sqrt{\kappa} \left(\sqrt{1-\varepsilon^2}(1-\kappa) +2 \varepsilon \sqrt{\kappa}\right)}g_1
-\frac{1+\kappa}{2 \kappa}g_0 +L(x_0-x_1)},\\
\leq &\left(\varepsilon+ \sqrt{1-{\varepsilon}^2}\frac{1-\kappa}{2\sqrt{\kappa}}\right)^2 \|{g_0}\|^2,
\end{align*}
where the last inequality follows from the sign of the coefficient:
\[\kappa \frac{2 \varepsilon \sqrt{\kappa}+(1-\kappa)\sqrt{1-\varepsilon^2} }{(1-\kappa)\sqrt{1-\varepsilon^2}}\geq 0.\]
Next, we prove the exact same guarantee as for the gradient norm, but in the case of distance to optimality $\|{x_1-x_*}\|^2$.

Let us consider the following constraints from \eqref{pep:els_x} with associated multipliers:
\begin{align*}
&{\frac{1}{1-\frac{\mu}{L}}\left(\frac{1}{L}\|{g_0}\|^2+\mu\|{x_0-x_*}\|^2-2\frac{\mu}{L}\langle{g_0},{x_0-x_*}\rangle\right)+\langle{g_0},{x_*-x_0}\rangle\leq 0} &&:\lambda_{0},\\
&{\frac{1}{1-\frac{\mu}{L}}\left(\frac{1}{L}\|{g_1}\|^2+\mu\|{x_1-x_*}\|^2-2\frac{\mu}{L}\langle{g_1},{x_1-x_*}\rangle\right)+\langle{g_1},{x_*-x_1}\rangle\leq 0} &&:\lambda_{1}, \\
&{\langle{g_1},{x_1-x_0}\rangle\leq 0} && :{\lambda_2},\\
&\begin{pmatrix}
\epsilon ||g_0||^2 & \langle{g_0},{g_1}\rangle\\
\inner{g_0}{g_1} & \epsilon ||g_1||^2
\end{pmatrix}\succeq 0 &&:S,
\end{align*}
with $S=\begin{pmatrix}
s_{11} & s_{12}\\ s_{12} & s_{22}
\end{pmatrix},$ and:
\begin{align*}
&\lambda_{0}=\frac{1-\kappa}{\mu}\left[1-2\varepsilon^2+\frac{\varepsilon\sqrt{1-\varepsilon^2}}{2\sqrt{\kappa}(1-\kappa)}(-1-\kappa^2+6\kappa)\right],\\
&\lambda_1=\frac{1}{\mu}-\frac{1}{L},\\
& {\lambda_2}=\frac{1}{\mu}+\frac{1}{L},\\
&L\mu s_{11}=\frac{3 \varepsilon}{2} - \frac{\varepsilon \left(\kappa + \frac{1}{\kappa}\right)}{4} + \frac{1-\kappa}{2 \sqrt{\kappa(1 - {\varepsilon}^2)}} - \frac{{\varepsilon}^2 \left(1 -  \kappa\right)}{\sqrt{\kappa(1 - {\varepsilon}^2)}},
&\\
&L\mu s_{22}= \frac{2  \sqrt{\kappa(1 - {\varepsilon}^2)}-\varepsilon \left(1 - \kappa\right)}{\left(1-{\varepsilon}^2\right) \left(1 - \kappa\right) + 2 \varepsilon \sqrt{\kappa(1 - {\varepsilon}^2)}}, &
\\
&L\mu s_{12}=
\frac{\varepsilon\, \left(1 - \kappa\right)}{2  \sqrt{\kappa(1 - {\varepsilon}^2)}} - 1.&
\end{align*}
{As in the case of the gradient norm, we proceed by reformulating the weighted sum of the constraints. For doing that, we first check nonnegativity of the weights $\lambda_0,\lambda_1,\lambda_2\geq 0$ and $S\succeq 0$.}

Similarly to the previous case, $s_{22}\geq 0\Leftrightarrow \varepsilon\in\left[\frac{-2\sqrt{\kappa}}{1+\kappa},\frac{2\sqrt{\kappa}}{1+\kappa}\right].$
 We therefore only need to check the sign of $\lambda_0$ in order to have the desired results (the $S\succeq 0$ requirement is the same as for the convergence in gradient norm, and the others are {easily} verified).
Concerning $\lambda_0$, we have
\[ \lambda_0\geq 0 \Leftrightarrow  \frac{\kappa-1}{\kappa+1} \leq \varepsilon\leq \frac{2\sqrt{\kappa}}{\kappa+1},\] with $\frac{\kappa-1}{\kappa+1} \leq  0$, {and hence $\lambda_0\geq 0$ in the region of interest}.

Aggregating the constraints with {the corresponding} multipliers yields:
\begin{align*}
\|{x_1-x_*}\|^2\leq &\left(\varepsilon+\sqrt{1-\varepsilon^2}\frac{1-\kappa}{2\sqrt{\kappa}}\right)^2 \|{x_0-x_*}\|^2 \\ &- \frac{2 \varepsilon \sqrt{\left(1-\varepsilon^2\right) \kappa}+\left(1-\varepsilon^2\right) (1-\kappa)}{\kappa(1-\kappa)}\times\\ &\normsq{\left(1-\frac{\varepsilon (1-\kappa)}{2 \sqrt{\left(1-\varepsilon^2\right) \kappa}}\right) \frac{g_0}{L} -\frac{1+\kappa}{2}(x_0-x_*)+\frac{1-\kappa}{2 \varepsilon \sqrt{\left(1-\varepsilon^2\right) \kappa}+\left(1-\varepsilon^2\right) (1-\kappa)}\frac{g_1}{L}},\\
\leq  &\left(\varepsilon+\sqrt{1-\varepsilon^2}\frac{1-\kappa}{2\sqrt{\kappa}}\right)^2 \|{x_0-x_*}\|^2,
\end{align*}
{where the last inequality follows from the sign of the coefficient
\[ \frac{2 \varepsilon \sqrt{\left(1-\varepsilon^2\right) \kappa}+\left(1-\varepsilon^2\right) (1-\kappa)}{\kappa(1-\kappa)}\geq 0.\]
This completes the proof.}
 \end{proof}

 \medskip
 Theorem~\ref{thm:ELS improvements} provides both tight and non-tight results, as follows:
\begin{enumerate}
	\item the result in function values cannot be improved, by~\cite[Example 5.2]{Klerk_Glineur_Taylor_2016};
	\item likewise, the result in gradient norm cannot be improved; we give an example proving this in Appendix \ref{appendix:B}.
	\item The result in distance to optimality is not tight.
\end{enumerate}
Our bounds on the rates satisfy
\[\left(\frac{1-\kappa+\varepsilon(1-\kappa)}{1+\kappa+\varepsilon(1-\kappa)}\right)^2\leq \left(\varepsilon + \sqrt{1-\varepsilon^2}\dfrac{1-\kappa}{2\sqrt{\kappa}}\right)^2,\]
where the left hand term is the optimal value of~\eqref{pep:els_f} and the right hand term is the optimal value of~\eqref{pep:els_g}.

 We may compare our results to known (classical) results with $\varepsilon = 0$. In that case, we have that the best possible rate for function value is given by (see~\cite[Theorem 5.2]{Klerk_Glineur_Taylor_2016}),
	\[f(x_1)-f(x_*)\leq \left(\frac{1-\kappa}{1+\kappa}\right)^2(f(x_0)-f(x_*)).\] By smoothness and strong convexity, we derive in a standard way for the gradient norm (the exact same reasoning holds for the distance to optimum)
\[ ||g(x_1)|| \leq \dfrac{1}{\sqrt{\kappa}} \left(\frac{1-\kappa}{1+\kappa}\right)  ||g(x_0)||,\] whereas Theorem \ref{thm:ELS improvements} provides the strictly better guarantee for $\kappa\in (0,1)$, namely:
\[||g(x_1)|| \leq \dfrac{1}{2\sqrt{\kappa}} \left( 1-\kappa \right)||g(x_0)||.\]
The above rates are valid when performing one iteration. Better rates can be guaranteed if more than one iteration is performed, which can be done
in the same framework. However, we do not pursue our
 investigations in that direction, as the subsequent analysis of
  Newton’s method only requires the best possible one-iteration inequalities, as provided by the above new improved bounds.

 One has the following variation on Theorem \ref{thm:ELS improvements} that deals with the case where $f \in \mathcal{F}_{\mu,L}(D)$ for some open convex set $D \subset \mathbb{R}^n$.
 \begin{theorem}
\label{thm:ELS improvements D}
Consider the inexact gradient method with exact line search applied to some twice continuously differentiable $f \in \mathcal{F}_{\mu,L}(D)$
where $D \subset \mathbb{R}^n$ is open and convex,
from a starting point $x_0 \in D$. Assume that $\left\{x \in \mathbb{R}^n \; | \; f(x) \le f(x_0)\right\} \subset D$.
If $\kappa := \frac{\mu}{L} \in (0,1]$ and $\varepsilon \in \left[0, \frac{2\sqrt{\kappa}}{1+\kappa}\right]$, one has
\begin{eqnarray*}
\|g(x_1)\| &\le & \left(\varepsilon+ \sqrt{1-{\varepsilon}^2}\frac{1-\kappa}{2\sqrt{\kappa}}\right)\|{g(x_0)}\|, \\
\|{x_1-x_*}\| & \le & \left(\varepsilon+ \sqrt{1-{\varepsilon}^2}\frac{1-\kappa}{2\sqrt{\kappa}}\right)\|{x_0-x_*}\|. \\
\end{eqnarray*}
\end{theorem}
\begin{proof}
The proof follows from the proof of Theorem \ref{thm:ELS improvements} after the following observations:
\begin{enumerate}
\item
The proof of the last two inequalities in Theorem \ref{thm:ELS improvements} only relies on the inequality \eqref{gradient inequality},
 which holds for any open, convex  $D \subset \mathbb{R}^n$, i.e.\ not only for $D = \mathbb{R}^n$, by Theorem \ref{thm:mu_L_H}.
\item
By the assumption on the level set of $f$, exact line search yields a point $x_1 \in D$, as required.
\end{enumerate}
\end{proof}

\medskip
Concerning tightness and comparisons with known results, the same remarks as for  Theorem~\ref{thm:ELS improvements} apply here.
Although the setting is nonstandard for first-order methods,
comparisons made for the case $D=\mathbb{R}^n$ are still valid as the worst-case bounds for gradient and distance are the same
 (i.e., results of Theorem~\ref{thm:ELS improvements D} already improve upon the literature/folklore knowledge in the simpler setting $D=\mathbb{R}^n$).

 Note that the first inequality in Theorem \ref{thm:ELS improvements} (convergence in
  function value) does not extend readily to all convex $D$, since its proof requires the inequality \eqref{eq:FmuL}.

\subsection{PEP with fixed step sizes}
We now state a result that is similar to Theorem \ref{thm:ELS improvements}, but deals with fixed step sizes instead of exact line search.
\begin{theorem}
\label{thm:FS improvements}
Consider the inexact gradient method with fixed step size $\gamma$ applied to some $f \in \mathcal{F}_{\mu,L}(\mathbb{R}^n)$.
If  $\varepsilon \in \left[0, \frac{2\mu}{L+\mu}\right]$, and $\gamma \in \left[0, \frac{2\mu-\varepsilon(L+\mu)}{(1-\varepsilon)\mu(L+\mu)}\right]$, one has
\begin{eqnarray*}
f(x_1)-f(x_*) &\le &  (1-(1-\varepsilon)\mu\gamma)^2(f(x_0)-f(x_*)), \\
\|g(x_1)\| &\le & (1-(1-\varepsilon)\mu\gamma) \norm{g(x_0)}, \\
\|{x_1-x_*}\| & \le & (1-(1-\varepsilon)\mu\gamma) \norm{x_0-x_*}. \\
\end{eqnarray*}
\end{theorem}
\begin{proof}
The proof is similar to that of Theorem \ref{thm:ELS improvements}, and is sketched in Appendix \ref{Appendix:A}.
\end{proof}

\medskip
 In Theorem~\ref{thm:FS improvements}, all results are provably tight. Indeed, one can easily verify that
 those bounds are achieved with equality (for all three cases: function, distance and gradient) on the one-dimensional minimization
  problem $\min_{x\in\mathbb{R}} f(x)$ with the quadratic function $f(x) = \dfrac{\mu}{2} x^2$ and the search
  direction $-g(x_0)(1-\varepsilon)= -\mu (1-\varepsilon) x_0$ (that satisfies the relative accuracy criterion~\eqref{eq:error}).
   Note that, in the case $\varepsilon=0$, one can therefore also recover all standard convergence guarantees that are tight
    on quadratics (see e.g.,~\cite{Taylor_Glineur_Hendrickx_2018} and the references therein).

\medskip
Note that, if $\gamma = \frac{2\mu-\varepsilon(L+\mu)}{(1-\varepsilon)\mu(L+\mu)}$, the factor $(1-(1-\varepsilon)\mu\gamma)$ that appears in the inequalities in Theorem \ref{thm:FS improvements} reduces to
\[
1-(1-\varepsilon)\mu\gamma = \frac{1-\kappa}{1+\kappa} + \varepsilon,
\]
where $\kappa = \mu/L$ as before.

Next, we again consider a variant constrained to an open, convex set $D \subset \mathbb{R}^n$.

\begin{theorem}
\label{Cor:F mu L error bounds D}
Assume $f \in \mathcal{F}_{\mu,L}(D)$ for some open convex set $D$, and $f$ twice continuously differentiable. Let $x_0 \in D$ so that $B(x_0,2\|x_0-x_*\|) \subset D$.
If $x_1 = x_0 - \gamma d$, with $\|d - g(x_0)\| \le \varepsilon\|g(x_0)\|$,  $\varepsilon \in \left[0, \frac{2\kappa}{1+\kappa}\right]$, and
\[
\gamma = \frac{2\mu-\varepsilon(L+\mu)}{(1-\varepsilon)\mu(L+\mu)},
\]
then
\begin{eqnarray*}
\|g_{x_0}(x_1)\| &\le & \left(\frac{1-\kappa}{1+\kappa}+\varepsilon\right)\|{g_{x_0}(x_0)}\|,\\
\|{x_1-x_*}\| & \le & \left(\frac{1-\kappa}{1+\kappa}+\varepsilon\right)\|{x_0-x_*}\|, \\
\end{eqnarray*}
where $\kappa= \mu/L$.
\end{theorem}
\begin{proof}
Note that the result follows from the proof of Theorem \ref{thm:FS improvements}, provided that $x_1 \in D$.
In other words, we need to show that the condition $x_1 \in D$ is a consequence of the hypotheses. This follows from:
 \begin{eqnarray*}
  \|x_1-x_0\| &\le & \|x_1-x_*\|_{x_0} + \|x_*-x_0\| \quad\mbox{ (triangle inequality)} \\
                    & \le & \left(\frac{1-\kappa}{1+\kappa}+\varepsilon\right)\|{x_0-x_*}\| + \|x_*-x_0\|  \quad\mbox{ (by Theorem \ref{thm:FS improvements})}\\
                    &\le& 2\|x_*-x_0\|  \quad\mbox{ (by $\varepsilon \le \frac{2\kappa}{1+\kappa}$)}, \\
 \end{eqnarray*}
which implies $x_1 \in D$ due to the assumption $B(x_0,2\|x_0-x_*\|) \subset D$.
\end{proof}

 The same remarks as for Theorem~\ref{thm:FS improvements} apply here:
 the results are tight on quadratics, as the worst-case bounds match those in the case $D=\mathbb{R}^d$.

\section{Implications for Newton's method for self-concordant $f$}
\label{sec:Implications for Newton}
Theorem \ref{Cor:F mu L error bounds D} has interesting implications when minimizing a self-concordant function $f$ with minimizer $x_*$ by
Newton's method.
The implications become clear when fixing a point $x_0 \in D_f$, and using the inner product $\langle \cdot,\cdot \rangle_{x_0}$.
Then the gradient at $x_0$ becomes $g_{x_0}(x_0) = H_{x_0}^{-1}(x_0)g(x_0)$, which is the opposite of the Newton step at $x_0$.
We will consider approximate Newton directions in the sense of \eqref{eq:error}, i.e.\ directions $-d$ that satisfy
$\|d-g_{x_0}(x_0)\|_{x_0} \le \varepsilon \|g_{x_0}(x_0)\|_{x_0}$, where $\varepsilon >0$ is given.
We only state results for the fixed step-length case, for later use.
Similar error bounds can be obtained using Theorem \ref{thm:ELS improvements D} for inexact Newton methods with exact line search,
that are used, e.g.\, in long step interior point methods with inexact search directions; see, e.g.\ \cite[\S 2.5.3]{Renegar2001a}.

\begin{corollary}
\label{Cor:SC error bounds FS}
Assume $f$ is self-concordant with minimizer $x_*$. Let $0< \delta < 1$ be given and $x_0 \in D_f$ so that $\|x_0-x_*\|_{x_0} \le \frac{1}{2}\delta$.
If  $x_1 = x_0 -\gamma d$,  where $\|d-g_{x_0}(x_0)\|_{x_0} \le \varepsilon \|g_{x_0}(x_0)\|_{x_0}$ with $\varepsilon \in \left[0, \frac{2(1-\delta)^4}{1+ (1-\delta)^4}\right]$, and
\[
\gamma = \frac{2(1-\delta)^4 - \varepsilon(1+(1-\delta)^4)}{(1-\varepsilon)(1-\delta)^2((1-\delta)^4 + 1)},
\]
then
\begin{eqnarray*}
\|g_{x_0}(x_1)\|_{x_0} &\le & \left(\frac{1-\kappa_\delta}{1+\kappa_\delta}+\varepsilon\right)\|{g_{x_0}(x_0)}\|_{x_0},\\
\|{x_1-x_*}\|_{x_0} & \le & \left(\frac{1-\kappa_\delta}{1+\kappa_\delta}+\varepsilon\right)\|{x_0-x_*}\|_{x_0}, \\
\end{eqnarray*}
where $\kappa_\delta = (1-\delta)^4$.

\end{corollary}
\begin{proof}
By Corollary \ref{thm:sc_vs_FmuL}, if we fix the inner product $\langle \cdot,\cdot\rangle_{x_0}$, then $ f \in \mathcal{F}_{\mu,L}(B_{x_0}(x_0,\delta))$ with
\begin{equation}
\label{eq:mu L delta}
\mu = (1-\delta)^2, \; L = \frac{1}{(1-\delta)^2}.
\end{equation}
 As a consequence $\kappa_\delta := \kappa = \mu/L  = (1-\delta)^4$. (We  use the notation $\kappa = \kappa_\delta$ to emphasize that
 $\kappa$ depends on $\delta$ (only).)
 The required result now follows from Theorem \ref{Cor:F mu L error bounds D}.
 \end{proof}


In view of our earlier remarks on tightness of the bounds in Theorem \ref{Cor:F mu L error bounds D},
it is important to note that the bounds in Corollary \ref{Cor:SC error bounds FS} are not tight in general.
The reason is that we only used the fact that, for a given $x_0 \in D_f$ and $\delta \in (0,1)$, one has $ f \in \mathcal{F}_{\mu,L}(B_{x_0}(x_0,\delta))$ for the
values of $\mu$ and $L$ as given in \eqref{eq:mu L delta}. This is weaker than requiring self-concordance of $f$, as the following example shows.

\begin{example}
Consider the univariate
$f(x) = \frac{1}{12}x^4$ with $D_f = (0,\infty)$. At $x_0 = 1$, one has $H(x_0) = 1$.
If we set $\delta = \frac{1}{2}$, \eqref{eq:mu L delta} yields $\mu = \frac{1}{4}$ and $L = 4$, and we have $B_{x_0}(x_0,\delta) = \left(\frac{1}{2},\frac{3}{2}\right)$.
Since $H_{x_0}(y) = y^2$ for all $y \in \mathbb{R}$, one has $\mu < H_{x_0}(y) < L$ if $y \in B_{x_0}(x_0,\delta)$, and therefore $ f \in \mathcal{F}_{\mu,L}(B_{x_0}(x_0,\delta))$.
On the other hand, $f$ is not self-concordant on its domain, since it does not satisfy the condition $|f'''(x)| \le 2 f''(x)^{3/2}$ if $x \in (0,1)$.
\end{example}

A final, but important observation is that the results in Corollary \ref{Cor:SC error bounds FS} remain
valid if we use the $\langle\cdot,\cdot\rangle_{x_*}$ inner product, as opposed to $\langle\cdot,\cdot\rangle_{x_0}$.
This implies that we (approximately) use the direction $-g_{x_*}(x_0) = -H^{-1}(x_*)g(x_0)$.
Such a direction may seem to be of no practical use, since $x_*$ is not known, but in the
next section we will analyze an interior point method that uses precisely such search directions.

For easy reference, we therefore state the worst-case convergence result when using  the $\langle\cdot,\cdot\rangle_{x_*}$ inner product.

\begin{corollary}
\label{Cor:SC error bounds FS2}
Assume $f$ is self-concordant with minimizer $x_*$. Let $ \delta \in(0, 1)$ be given and $x_0 \in D_f$ so that $\|x_0-x_*\|_{x_*} \le \frac{1}{2} \delta$.
If $x_1 = x_0 -\gamma d$, where $\|d-g_{x_*}(x_0)\|_{x_*} \le \varepsilon \|g_{x_*}(x_0)\|_{x_*}$ with
$\varepsilon \in \left[0,\frac{2(1-\delta)^4}{1+ (1-\delta)^4}\right]$, and step size
\begin{equation}
\label{eq:choice of gamma}
\gamma = \frac{2(1-\delta)^4 - \varepsilon(1+(1-\delta)^4)}{(1-\varepsilon)(1-\delta)^2((1-\delta)^4 + 1)},
\end{equation}
then
\begin{eqnarray*}
\|g_{x_0}(x_1)\|_{x_*} &\le & \left(\frac{1-\kappa_\delta}{1+\kappa_\delta}+\varepsilon\right)\|{g_{x_0}(x_0)}\|_{x_*},\\
\|{x_1-x_*}\|_{x_*} & \le & \left(\frac{1-\kappa_\delta}{1+\kappa_\delta}+\varepsilon\right)\|{x_0-x_*}\|_{x_*}, \\
\end{eqnarray*}
where $\kappa_\delta = (1-\delta)^4$.

\end{corollary}

We may compare Corollaries
\ref{Cor:SC error bounds FS} and \ref{Cor:SC error bounds FS2} to the following results that may be obtained from standard interior point analysis.

\begin{theorem}[Based on Theorems 1.6.2 and 2.2.3 in \cite{Renegar2001a}]
	\label{lemma:ImproveDistanceZ}
	Let $f$ be a self-concordant function with minimizer $x_*$, and let $x_0 \in B_{x_0}(x_*,1)$. Define $x_1 = x_0 - \gamma[ H(x_0)^{-1}g(x_0) + e(x_0)]$ for some $\gamma \in (0,1)$, where $e(x_0)$ denotes an error in the Newton direction at the point $x_0$.
	If $\| e(x_0) \|_{x_0} \leq \varepsilon \|H(x_0)^{-1}g(x_0)\|_{x_0}$, then
	\begin{equation}
\label{Newton SC bound}
		\| x_1 - x_* \|_{x_0} \leq \frac{(1 - \gamma + \gamma^2 \varepsilon)\|x_0-x_*\|_{x_0} + \gamma \|x_0-x_*\|_{x_0}^2}{\gamma(1 - \|x_0-x_*\|_{x_0})}.
	\end{equation}
Similarly, if we define instead $x_1 = x_0 - \gamma[ H(x_*)^{-1}g(x_0) + e(x_0)]$, i.e.\ replace $H(x_0)$ by $H(x_*)$ in the definition of $x_1$, then
	\begin{equation}
\label{Newton-type SC bound}
		\| x_1 - x_* \|_{x_*} \leq \frac{(1 - \gamma + \gamma^2 \varepsilon)\|x_0-x_*\|_{x_*} + \gamma \|x_0-x_*\|_{x_*}^2}{\gamma(1 - \|x_0-x_*\|_{x_*})},
	\end{equation}
under the assumption $x_0 \in B_{x_*}(x_*,1)$.
\end{theorem}
Note that the only difference between the inequalities  \eqref{Newton SC bound} and \eqref{Newton-type SC bound} is the choice of local norm.

To compare Theorem \ref{lemma:ImproveDistanceZ} to Corollaries  \ref{Cor:SC error bounds FS} and \ref{Cor:SC error bounds FS2}, we present a plot of the
respective upper bounds in Figure \ref{fig:Comparison} for for different values of $\varepsilon$. The value of the step size $\gamma$ is as in
\eqref{eq:choice of gamma} with $\delta = 2\|x_0-x_*\|_{x_*}$.

\begin{figure}[h!]
	\begin{tikzpicture}
		\begin{axis}[plotOptions, title={$\varepsilon = 0$}, small,	domain=0:0.25,	ylabel near ticks]
			\addplot+[dashed, mark=none]{-(x^2*(8*x^3 - 12*x^2 + 4*x + 1))/((2*x - 1)^2*(x - 1))};
			\addplot+[mark=none]{x*(1-(1-2*x)^4)/(1+(1-2*x)^4)};
			\legend{Thm. \ref{lemma:ImproveDistanceZ}, Cor. \ref{Cor:SC error bounds FS2}}
		\end{axis}
	\end{tikzpicture}\hfill
	\begin{tikzpicture}
		\begin{axis}[plotOptions, title={$\varepsilon = \frac{1}{2} \frac{(1- 2\| x_0-x_* \|_{x_*})^4}{1 + (1-2 \| x_0-x_*\|_{x_*})^4}$}, small,	domain=0:0.25,	ylabel near ticks]
			\addplot+[dashed, mark=none]{(x*((2*x - 1)^4/(2*((2*x - 1)^4 + 1)) - 1)*((2*x - 1)^4 + 1)*(8192*x^12 - 43008*x^11 + 98304*x^10 - 125440*x^9 + 97024*x^8 - 48256*x^7 + 18688*x^6 - 8416*x^5 + 4256*x^4 - 1696*x^3 + 448*x^2 - 60*x + 9))/(6*(2*x - 1)^2*(x - 1)*(8*x^4 - 16*x^3 + 12*x^2 - 4*x + 1)*(16*x^4 - 32*x^3 + 24*x^2 - 8*x + 3)^2)};
			\addplot+[mark=none]{x*(1-(1-2*x)^4/2)/(1+(1-2*x)^4)};
			\legend{Thm. \ref{lemma:ImproveDistanceZ}, Cor. \ref{Cor:SC error bounds FS2}}
		\end{axis}
	\end{tikzpicture}
	
	\begin{tikzpicture}
		\begin{axis}[plotOptions, title={$\varepsilon =  \frac{(1- 2\| x_0-x_* \|_{x_*})^4}{1 + (1-2 \| x_0-x_*\|_{x_*})^4}$}, small,	domain=0:0.25,	ylabel near ticks]
			\addplot+[dashed, mark=none]{-(x*(256*x^8 - 960*x^7 + 1536*x^6 - 1360*x^5 + 736*x^4 - 256*x^3 + 56*x^2 - 6*x + 1))/(2*(2*x - 1)^2*(8*x^5 - 24*x^4 + 28*x^3 - 16*x^2 + 5*x - 1))};
			\addplot+[mark=none]{x*(1/(1+(1-2*x)^4) )};
			\legend{Thm. \ref{lemma:ImproveDistanceZ}, Cor. \ref{Cor:SC error bounds FS2}}
		\end{axis}
	\end{tikzpicture}\hfill
	\begin{tikzpicture}
		\begin{axis}[plotOptions, title={$\varepsilon = \frac{3}{2} \frac{(1- 2\| x_0-x_* \|_{x_*})^4}{1 + (1-2 \| x_0-x_*\|_{x_*})^4}$}, small,	domain=0:0.25,	ylabel near ticks]
			\addplot+[dashed, mark=none]{(x*((3*(2*x - 1)^4)/(2*((2*x - 1)^4 + 1)) - 1)*((2*x - 1)^4 + 1)*(8192*x^12 - 51200*x^11 + 147456*x^10 - 258560*x^9 + 306432*x^8 - 257152*x^7 + 155392*x^6 - 67296*x^5 + 20320*x^4 - 4000*x^3 + 416*x^2 + 4*x + 3))/(2*(2*x - 1)^2*(x - 1)*(8*x^4 - 16*x^3 + 12*x^2 - 4*x + 1)*(- 16*x^4 + 32*x^3 - 24*x^2 + 8*x + 1)^2)};
			\addplot+[mark=none]{x*(1+(1-2*x)^4/2)/(1+(1-2*x)^4)};
			\legend{Thm. \ref{lemma:ImproveDistanceZ}, Cor. \ref{Cor:SC error bounds FS2}}
		\end{axis}
	\end{tikzpicture}
	\caption{Upper bounds on $\| x_1 - x_* \|_{x_*}$ from Theorem \ref{lemma:ImproveDistanceZ} and Corollary \ref{Cor:SC error bounds FS2}.}
	\label{fig:Comparison}
\end{figure}

A few remarks on Figure \ref{fig:Comparison}:
\begin{enumerate}
\item
Although the figure only compares inequality \eqref{Newton-type SC bound} in Theorem \ref{lemma:ImproveDistanceZ} to the bound
in Corollary  \ref{Cor:SC error bounds FS2}, the exact same plots remain valid when comparing
the Newton direction bounds, namely inequality \eqref{Newton SC bound} in Theorem \ref{lemma:ImproveDistanceZ} to the bound
in Corollary  \ref{Cor:SC error bounds FS}. The only difference is the scaling on the axes, since one should then switch from
the $\|\cdot\|_{x_*}$ norm to the $\|\cdot\|_{x_0}$ norm. In this case the value $\gamma$ is still given by \eqref{eq:choice of gamma}, but with $\delta = 2\|x_0-x_*\|_{x_0}$.
\item
It is clear that our new bounds in Corollary  \ref{Cor:SC error bounds FS2} (and Corollary  \ref{Cor:SC error bounds FS}) improve on the known bounds
in most cases. Even when $\varepsilon = 0$, we still improve if  $\| x_0 - x_* \|_{x_*}$ is sufficiently large. As $\varepsilon$ grows, our bounds clearly
improve on those in Theorem \ref{lemma:ImproveDistanceZ}.
\item
In the figure, our new error bound remains bounded as the initial distance $\| x_0 - x_* \|_{x_*}$ approaches $1$, but this is not the case for
the bound from Theorem \ref{lemma:ImproveDistanceZ}. Thus our new results capture a desirable feature of the convergence near the boundary of the Dikin ellipsoid.
\end{enumerate}

\section{Complexity of a short step interior point method using inexact search directions}
\label{sec:inexactIPM}
We now sketch a proof of how to bound the worst-case iteration complexity of a short step interior point method using
inexact search directions.

Given a convex body $\mathcal{K}\subset \mathbb{R}^n$ and a vector $\hat\theta \in \mathbb{R}^n$, we consider the convex optimization problem
\begin{equation}
\label{prob:convex}
\min_{x \in \mathcal{K}} \hat \theta^\tsp x.
\end{equation}

A subclass of self-concordant functions, that play a key role in interior point analysis, are the so-called self-concordant barriers.
\begin{definition}[Self-concordant barrier]
A self-concordant function  $f$ is called a $\vartheta$-self-concordant barrier
if there is a finite value $\vartheta \ge 1$ given by
\[
  \vartheta := \sup_{x\in D_f} \|g_x(x)\|_x^2.
\]
\end{definition}
We will assume that we know a self-concordant barrier function with domain given by the interior of $\mathcal{K}$, say $f_{\mathcal{K}}$.

The key observation is that one may analyse the complexity of interior point methods by only analysing
the progress during one iteration; see e.g.\ \cite[\S2.4]{Renegar2001a}. Thus our analysis of the previous section may be applied readily.
At each interior point iteration, one approximately minimizes a self-concordant function of the form
\begin{equation}
\label{def:f}
f(x) = \eta\hat\theta^\top x + f_{\mathcal{K}}(x),
\end{equation}
where $\eta > 0$ is a fixed parameter. We denote its unique minimizer
 by $x(\eta)$, and call it the point on the \emph{central path} corresponding to $\eta$.
Subsequently, in the next interior point iteration, the value of $\eta$ is increased, and the process is repeated.

We may now state two variants ($A$ and $B$) of a short step, interior point method using inexact search directions (see Algorithm \ref{alg}).
Variant $A$ corresponds to the short step interior point  method analysed by Renegar \cite[\S2.4.2]{Renegar2001a}, but
allows for inexact Newton directions. Variant $B$ captures the framework of the interior point method of
Abernethy and Hazan \cite[Appendix D, supplementary material]{Abernethy_Hazan_2016}, to be discussed in Section \ref{sec:Analysis of the method of Abernathy-Hazan}.

\begin{algorithm}
\caption{Short step interior point method using
inexact directions (variants $A$ and $B$) \label{alg}}
\begin{algorithmic}
 \STATE{Tolerances: $\varepsilon > 0$ (for search direction error), $\bar \epsilon > 0$ (for stopping criterion)}
 \STATE{Proximity to central path parameter: $\delta \in (0,1)$}
 \STATE{Barrier parameter for $f_{\mathcal{K}}$: $\vartheta \ge 1$}
 \STATE{ Objective vector $\hat \theta \in \mathbb{R}^n$}
 \STATE{Given an $x_0\in \mathcal{K}$ and $\eta_0 >0$ such that $\|x_0- x(\eta_0)\|_{x(\eta_0)} \le \frac{1}{2}\delta$ (variant $A$) or
 $\|x_0- x(\eta_0)\|_{x_0} \le \frac{1}{2}\delta$ (variant $B$).}
\STATE{Set the step size $\gamma = \frac{2(1-\delta)^4 - \varepsilon(1+(1-\delta)^4)}{(1-\varepsilon)(1-\delta)^2((1-\delta)^4 + 1)}$}
    \STATE{Iteration: $k=0$}
 \WHILE{$\frac{\vartheta}{\eta_k} > \frac{5}{6}\bar\epsilon$}
  \STATE{ compute $ d$ that satisfies $\|d-g_{x_k}(x_k)\|_{x_k} \le \varepsilon \|g_{x_k}(x_k)\|_{x_k}$ (variant $A$)
   or $\|d-g_{x(\eta)}(x_k)\|_{x(\eta)} \le \varepsilon \|g_{x(\eta)}(x_k)\|_{x(\eta)}$ (variant $B$)}
  \STATE{$x_{k+1} = x_k - \gamma  d$}
  \STATE{$\eta_{k+1} = \left(1+ \frac{1}{32\sqrt{\vartheta}}\right)\eta_k$}
  \STATE{$k \leftarrow k+1$}
  \ENDWHILE
\RETURN {$x_k$ an $\bar \epsilon$-optimal solution to $\min_{x\in \mathcal{K}} \hat \theta^\top x$}
\end{algorithmic}
\end{algorithm}

We will show the following worst-case iteration complexity result.

\begin{theorem}
\label{thm:complexity bound IPM iterations}
Consider Algorithm \ref{alg} with the following input parameter settings:
 $\bar \epsilon > 0$,   $0 \le \varepsilon  \le \frac{1}{6}$, and $\delta = \frac{1}{4}$.
Both variants of the algorithm then  terminate
after at most
\[
N =\left\lceil 40\sqrt{\vartheta}\ln\left(\frac{6\vartheta}{5\eta_0 \bar\epsilon}\right) \right\rceil
\]
iterations. The result is an $x_N \in \mathcal{K}$ such that
\[
 \hat \theta^\top x_N  - \min_{x\in \mathcal{K}} \hat \theta^\top x \le \bar \epsilon.
\]
\end{theorem}
\begin{proof}
The proof follows the usual lines of analysis of short step interior point methods; in particular we will repeatedly refer to
Renegar  \cite[\S2.4]{Renegar2001a}.
We only analyse variant $B$ of Algorithm \ref{thm:complexity bound IPM iterations}, as the analysis of
variant $A$ is similar, but simpler.

We only need to show that, at the start of each iteration $k$, one has
\[
\|x_k- x(\eta_k)\|_{x(\eta_k)} \le \frac{1}{2}\delta.
\]
Since on the central path one has $ \hat \theta^\top x(\eta)  - \min_{x\in \mathcal{K}} \hat \theta^\top x  \le \vartheta/\eta$, the required result will
then follow in the usual way (following the proof of relation (2.18) in  \cite[p.\ 47]{Renegar2001a}).

Without loss of generality we therefore only consider the first iteration, with a given
$x_0\in \mathcal{K}$ and $\eta_0 >0$ such that
$\|x_0- x(\eta_0)\|_{x(\eta_0)} \le \frac{1}{2}\delta$, and proceed to show that
$\|x_1- x(\eta_1)\|_{x(\eta_1)} \le \frac{1}{2}\delta$.

First, we bound the difference between the successive `target' points on the central path, namely $x(\eta_0)$ and $x(\eta_1)$, where
$\eta_1 = \left(1+ \frac{\alpha}{\sqrt{\vartheta}}\right)\eta_0$ with $\alpha = 1/32$. By the same argument as in \cite[p.\ 46]{Renegar2001a}, one obtains:
\begin{eqnarray*}
\|x(\eta_1) - x(\eta_0)\|_{x(\eta_0)} &\le& \alpha + \frac{3\alpha^2}{(1-\alpha)^3}\\
                                      &\le& 0.0345 \mbox{ for $\alpha = 1/32$}.
\end{eqnarray*}
Moreover, by Corollary \ref{Cor:SC error bounds FS2},
\begin{eqnarray*}
\|{x_1-x(\eta_0)}\|_{x(\eta_0)}  &\le&  \left(\frac{1-(1-\delta)^4}{1+(1-\delta)^4}+\varepsilon\right)\|{x_0-x(\eta_0})\|_{x(\eta_0)} \\
                                  &\le & 0.6860\cdot \frac{1}{2} \delta \le  0.0857.
\end{eqnarray*}
Using the triangle inequality,
\begin{eqnarray*}
\|{x_1-x(\eta_1)}\|_{x(\eta_0)}  &\le& \|{x_1-x(\eta_0)}\|_{x(\eta_0)} + \|{x(\eta_1)-x(\eta_0)}\|_{x(\eta_0)} \\
& \le & 0.0857 + 0.0345 = 0.1202.
\end{eqnarray*}
Finally, by the definition of self-concordance, one has
\[
\|{x_1-x(\eta_1)}\|_{x(\eta_1)} \le \frac{\|{x_1-x(\eta_1)}\|_{x(\eta_0)}}{1-\|{x(\eta_0)-x(\eta_1)}\|_{x(\eta_0)}} \le \frac{0.1202}{1-0.0345} \le 0.1245 < \frac{1}{2}\delta,
\]
as required.
\end{proof}

\medskip
It is insightful to note that, in the proof of Theorem \ref{thm:complexity bound IPM iterations}, it would not suffice to use the classical bound from Theorem \ref{lemma:ImproveDistanceZ}.
Indeed, we used $\|{x_1-x(\eta_0)}\|_{x(\eta_0)} \le 0.0857$ in the proof, obtained from our new bound in Corollary \ref{Cor:SC error bounds FS2}.
If we had used Theorem \ref{lemma:ImproveDistanceZ} instead, we would only obtain $\|{x_1-x(\eta_0)}\|_{x(\eta_0)} \le 0.1042$ (by using $\gamma = 0.67$), which would be too weak to complete the argument.
Of course, one could prove a variation on Theorem \ref{thm:complexity bound IPM iterations} by using Theorem \ref{lemma:ImproveDistanceZ}
and smaller values of $\delta$ and $\varepsilon$. Having said that, it is clear that our analysis adds in a meaningful way to the classical
 interior point analysis,
removing the need to use weaker parameter values.

\subsection{Analysis of the method of Abernathy-Hazan}
\label{sec:Analysis of the method of Abernathy-Hazan}

Abernathy and Hazan \cite{Abernethy_Hazan_2016} describe an interior point method to solve the convex optimization problem
\eqref{prob:convex}
if one only has access to a membership oracle for $\mathcal{K}$ (see  Abernethy and Hazan \cite[Appendix D, supplementary material]{Abernethy_Hazan_2016}).
As mentioned earlier, it falls within the framework of variant B of Algorithm \ref{alg} above.

 This method has generated  recent interest, since it is closely related to
a simulated annealing algorithm, and may be implemented by only sampling from $\mathcal{K}$.
Polynomial-time complexity of certain simulated annealing methods for
convex optimization was first shown by Kalai and Vempala \cite{Kalai-Vempala 2006}, and the link with interior point methods
casts light on their result.

The interior point method in question
used the so-called entropic (self-concordant) barrier function, introduced by Bubeck and Eldan \cite{Bubeck_Eldan_2014}, and we first review the necessary background.

\subsubsection{Background on the entropic barrier method}
The following discussion  is condensed from  \cite{Abernethy_Hazan_2016}.

The method is best described by considering the Boltzman probability distribution on $\mathcal{K}$:
\[
 P_\theta(x) := \exp(- \theta^\top x - A(\theta)) \quad \text{ where } \quad A(\theta) := \ln \int_K \exp(-\theta^\top x')dx',
\]
where $\theta = \eta \hat \theta$ for some fixed parameter $\eta > 0$. We write $X \sim P_\theta$ if the random variable $X$ takes values in $\mathcal{K}$
according to the Boltzman probability distribution on $\mathcal{K}$ with density $P_\theta$.

The convex function $A(\cdot)$ is known as the \emph{log partition function}, and has derivatives:
\begin{eqnarray*}
	\nabla A(\theta) & = & - \mathbb{E}_{X \sim P_\theta}[X]  \\
	\nabla^2 A(\theta) & = & \mathbb{E}_{X \sim P_\theta}[(X - \mathbb{E}_{X \sim P_\theta}[X])(X - \mathbb{E}_{X \sim P_\theta}[X])^\top].
\end{eqnarray*}

The Fenchel conjugate of $A(\theta)$ is
\[
 A^*(x) := \sup_{\theta \in \mathbb{R}^n} \theta^\top x - A(\theta).
\]
  The domain of $A^*(\cdot)$ is precisely the space of gradients of $A(\cdot)$,  and this is the set $\text{int}(-\mathcal{K})$.

The following key result shows that $A^*$ provides a self-concordant barrier for the set $\mathcal{K}$.
 \begin{theorem}[Bubeck-Eldan \cite{Bubeck_Eldan_2014}]
	The function $x \mapsto A^*(-x)$ is a $\vartheta$-self-concordant barrier function on $\mathcal{K}$ with $\vartheta \leq n(1 + o(1))$.
\end{theorem}
The function $x \mapsto A^*(-x)$ is denoted by $A^*_{-}(\cdot)$ and called the {entropic barrier} for $\mathcal{K}$.

At every step of the associated interior point method, one wishes to minimize (approximately) a self-concordant function of the
form \eqref{def:f}, where we now have the barrier function $f_{K}(x) = A^*_{-}(x)$.

In keeping with our earlier notation for performance estimation,
 we denote the minimizer of $f$ on $\mathcal{K}$ by $x_*$ (as opposed to $x(\eta)$). Thus $x_*$ is the point on the central path corresponding to
the parameter $\eta$.
We also assume a current iterate $x_0 \in \mbox{int}(\mathcal{K})$ is available
so that $\|x_* -x_0\|_{x_*} = \frac{1}{2}\delta < \frac{1}{2}$.

Abernathy and Hazan \cite[Appendix D, supplementary material]{Abernethy_Hazan_2016} propose to use the following direction to minimize $f$:
\begin{equation}
\label{eq:search direction Hazan}
-d = -\nabla^2 f (x_*)^{-1}\nabla f(x_0).
\end{equation}
The underlying idea is that $\nabla^2 f (x_*)^{-1}$ may be approximated  to any given accuracy through sampling, based on the following result.

\begin{lemma}[\cite{Bubeck_Eldan_2014}]
\label{thm:Hessian sampling}
One has
\[
\nabla^2 f (x_*)^{-1} = \nabla^2 A(\theta) = \mathbb{E}_{X \sim P_\theta}[(X - \mathbb{E}_{X \sim P_\theta}[X])(X - \mathbb{E}_{X \sim P_\theta}[X])^\top],
\]
where $\theta = \eta \hat \theta$.
\end{lemma}
The proof follows immediately from the relationship between the Hessians of a convex function and its conjugate, as  given in \cite{Crouzeix_1977}.

Thus we may approximate $\nabla^2 f (x_*)^{-1}$ by an empirical covariance matrix as follows.
If $X_i  \sim P_\theta$ $(i = 1,\ldots,N)$ are i.i.d., then we define the
associated estimator of the covariance matrix of the $X_i$'s as
\begin{equation}
\label{empirical covariance}
\hat\Sigma := \frac{1}{N}\sum_{i=1}^N (X_i - \bar X)(X_i - \bar X)^\top \quad \mbox{ where } \bar X = \frac{1}{N} \sum_{i=1}^N X_i.
\end{equation}
The estimator $\hat \Sigma$ is known as the empirical covariance matrix, and it may be observed by sampling $X  \sim P_\theta$.
This may be done efficiently:
for example, Lov\'asz and Vempala \cite{Lovasz_Vempala_2007} showed that one may sample (approximately) from log-concave distributions on compact bodies in polynomial time,
by using the Markov-chain Monte-Carlo sampling method called hit-and-run, introduced by Smith \cite{Smith 1984}.

The following concentration result (i.e.\ error bound) is known for the
empirical covariance matrix. We state this result to motivate our framework of analysis only --- we will not use it.

\begin{theorem}[cf.\ Theorems 4.1 and 4.2  in \cite{Adamczak_et_al_2010}]
\label{th:empirical covariance}
Assume $\epsilon \in (0,1)$ and  $X_i  \sim P_\theta$ $(i = 1,\ldots,N)$ are i.i.d.,
\begin{equation}
\label{eq:sigma}
\Sigma = \mathbb{E}_{X \sim P_\theta}\left[(X - \mathbb{E}_{X \sim P_\theta}[X])(X - \mathbb{E}_{X \sim P_\theta}[X])^\top\right]
\end{equation}
 is the covariance matrix,
 and $\hat \Sigma$ is the empirical covariance
 matrix in \eqref{empirical covariance}.
Then there exist absolute constants $c>0$ and $C>0$, such that, for $N \ge C\frac{\|\Sigma\|^2}{\epsilon^2}\log^2\left(\frac{2\|\Sigma\|^2}{\epsilon^2}\right)n$,
the following holds with probability at least $1 - \exp(-c \sqrt{n})$:
\begin{eqnarray}
 (1-\epsilon ) y^\top \hat\Sigma y            &\le   &       y^\top\Sigma y \le (1+\epsilon) y^\top \hat\Sigma y \quad\quad\quad\quad  \; \forall  y \in \mathbb{R}^n \label{isotropy1}\\
 (1-\epsilon) y^\top \hat\Sigma^{-1} y           & \le  &        y^\top\Sigma^{-1}y \le (1+\epsilon) y^\top \hat\Sigma^{-1} y \quad\quad \forall  y \in \mathbb{R}^n.  \label{isotropy2}
\end{eqnarray}
\end{theorem}
The exact details of hit-and-run sampling are outside the scope of this paper.
 For simplicity, we will therefore assume, in what follows,
the availability of an approximate covariance matrix $\hat\Sigma $ that satisfies \eqref{isotropy1} and \eqref{isotropy2}.
In other words, Theorem \ref{th:empirical covariance} only serves to motivate our assumption, we will not use it in our analysis.
We give a full analysis of the sampling process in the separate work \cite{Badenbroek_DeKlerk2018}, where we show how to
find a sample covariance matrix $\hat\Sigma $ that satisfies \eqref{isotropy1} and \eqref{isotropy2}.

\subsubsection{Analysis of the approximate direction in the  Abernethy-Hazan algorithm}
We can now show that an approximation of the search direction of Abernethy-Hazan \eqref{eq:search direction Hazan} satisfies
our `approximate negative gradient' condition \eqref{eq:error}.

\begin{theorem}
\label{thm:approximate direction}
Let $\epsilon > 0$ be given, the covariance matrix $\Sigma$ as in \eqref{eq:sigma}, and a symmetric matrix $\hat \Sigma$ that
approximates $\Sigma$ as in \eqref{isotropy1} and \eqref{isotropy2}. Further, let $f$ be as in \eqref{def:f} with minimizer $x_*$
on a given convex body $\mathcal{K}$.
Then the direction $-d = - \hat \Sigma \nabla f(x_0)$ at $x_0 \in \mathcal{K}$ satisfies
\[
 \|\nabla^2 f(x_*)^{-1}\nabla f(x_0) - d\|_{x_*} \le \sqrt{\frac{2\epsilon}{1-\epsilon}} \|\nabla^2 f(x_*)^{-1}\nabla f(x_0)\|_{x_*}.
\]
In other words, one has
$\|g_{x_*}(x_0) - d\|_{x_*} \le \varepsilon \|g_{x_*}(x_0) \|_{x_*}$ where $\varepsilon = \sqrt{\frac{2\epsilon}{1-\epsilon}}$, i.e.\ condition \eqref{eq:error}
holds for the inner product $\langle \cdot,\cdot\rangle_{x_*}$,
when the reference inner product $\langle \cdot,\cdot\rangle$ is the Euclidean dot product.
\end{theorem}
\begin{proof}
We fix the reference inner product $\langle \cdot,\cdot\rangle$ as the Euclidean dot product,
so that $H(x_*) = \nabla^2 f(x_*) = \Sigma^{-1}$ and $g(x_0) = \nabla f(x_0)$.
One has
\begin{eqnarray*}
\|H^{-1}(x_*)g(x_0) - d\|^2_{x_*} & = & \langle H^{-1}(x_*)g(x_0) - \hat \Sigma g(x_0), H^{-1}(x_*)g(x_0) - \hat \Sigma g(x_0)\rangle_{x_*} \\
                            & = & \langle H^{-1}(x_*)g(x_0) - \hat \Sigma g(x_0), g(x_0) - H(x_*)\hat \Sigma g(x_0)]\rangle \\
            & = & g(x_0)^\top H^{-1}(x_*)g(x_0) - 2g(x_0)^\top \hat \Sigma g(x_0) + [\hat \Sigma g(x_0)]^\top H(x_*)[\hat \Sigma g(x_0)] \\
            & \le & (1+\epsilon)g(x_0)^\top \hat \Sigma g(x_0) - 2g(x_0)^\top \hat \Sigma g(x_0) + (1+\epsilon)[\hat \Sigma g(x_0)]^\top \hat \Sigma^{-1}[\hat \Sigma g(x_0)] \\
& = & 2\epsilon \cdot g(x_0)^\top \hat \Sigma g(x_0),
\end{eqnarray*}
where the inequality is from \eqref{isotropy1} and \eqref{isotropy2}. Finally, using \eqref{isotropy1} once more, one obtains
\begin{eqnarray*}
\|H^{-1}(x_*)g(x_0) - d\|^2_{x_*} &\le& \frac{2\epsilon}{1-\epsilon} g(x_0)^\top H^{-1}(x_*) g(x_0) \\
&= & \frac{2\epsilon}{1-\epsilon}\|g_{x_*}(x_0)\|^2_{x_*},
\end{eqnarray*}
as required.
\end{proof}

We need to consider another variant of the search direction in \eqref{eq:search direction Hazan},
since $\nabla f(x_0)$ will not be available exactly in general. Indeed, one can only obtain $\nabla f(x_0)$ approximately via the relation
\[
\nabla f(x_0) = \eta \hat \theta +  \nabla A^*_{-}(x_0) = \eta \hat \theta + \arg\max_{\theta \in \mathbb{R}^n} \left[-\theta^\top x_0 - A(\theta)\right],
\]
where the last equality follows from the relationship between first derivatives of conjugate functions.

Thus $\nabla f(x_0)$ may be approximated by solving an unconstrained concave maximization problem in $\theta$ approximately, and,
 for this purpose, one may use the derivatives of
$A(\theta)$ as given above.

In particular, we will assume that we have available a $\tilde g(x_0) \approx \nabla f(x_0)$ in the sense that
\begin{equation}
\| \tilde g_{x_*}(x_0) -  g_{x_*}(x_0)\|_{x_*} \le \epsilon' \|g_{x_*}(x_0)\|_{x_*},
\label{eq:tilde g}
\end{equation}
where $\tilde g_{x_*}(x_0) := \Sigma \tilde g(x_0)$, $g_{x_*}(x_0) = \Sigma \nabla f(x_0)$ as before, and $\epsilon' > 0$ is given.
To motivate our assumption, we note that the function $\theta \mapsto -\theta^\top x_0 - A(\theta)$ is self-concordant \cite{Bubeck_Eldan_2014},
and we may therefore use Corollary \ref{Cor:SC error bounds FS} to bound the complexity of approximating its maximizer.
As with the Hessian approximation, we again omit the (lengthy) details; see Section \ref{sec:conclusion} for a further discussion of our assumptions.

Thus we will consider the search direction
\begin{equation}
-\tilde d := -\hat \Sigma \tilde g(x_0) \approx -\Sigma \nabla f(x_0).
\label{eq:tilde d}
\end{equation}
\begin{corollary}
Under the assumptions of Theorem \ref{thm:approximate direction}, define for a given $\epsilon' > 0$, the direction
$-\tilde d$ at $x_0 \in D_f$ as in \eqref{eq:tilde d}, where $\tilde g(x_0) \approx \nabla f(x_0)$ satisfies \eqref{eq:tilde g}.
Then one has
\[
\|\tilde d - g_{x_*}(x_0)\|_{x_*} \le \left(\epsilon'\cdot \sqrt{\frac{1+\epsilon}{1-\epsilon}}+\sqrt{\frac{2\epsilon}{1-\epsilon}}\right)\|g_{x_*}(x_0)\|_{x_*}.
\]
In other words, one has
$\|g_{x_*}(x_0) - \tilde d\|_{x_*} \le \varepsilon \|g_{x_*}(x_0) \|_{x_*}$ where $\varepsilon = \epsilon'\cdot \sqrt{\frac{1+\epsilon}{1-\epsilon}}+\sqrt{\frac{2\epsilon}{1-\epsilon}}$, i.e.\ condition \eqref{eq:error}
holds for the inner product $\langle \cdot,\cdot\rangle_{x_*}$,
when the reference inner product $\langle \cdot,\cdot\rangle$ is the Euclidean dot product.
\end{corollary}
\begin{proof}
Recall the notation $d = \hat \Sigma \nabla f(x_0)$ from Theorem \ref{thm:approximate direction}, and note that, by definition,
\begin{eqnarray*}
\| \tilde d - d\|^2_{x_*} &=& \|\hat \Sigma\left(\tilde g(x_0)-g(x_0) \right)\|_{x_*}^2 \\
&=& \langle \hat  \Sigma(\tilde g(x_0)-g(x_0), \Sigma^{-1}\hat \Sigma(\tilde g(x_0)-g(x_0)\rangle \\
&\le& (1+\epsilon)(\tilde g(x_0)-g(x_0))^\top \hat \Sigma (\tilde g(x_0)-g(x_0)) \quad\quad\mbox{ (by \eqref{isotropy2})} \\
&\le& \left(\frac{1+\epsilon}{1-\epsilon}\right)\|\tilde g_{x_*}(x_0)-g_{x_*}(x_0)\|^2_{x_*} \quad\quad\quad\quad\quad\quad\mbox{ (by \eqref{isotropy1})}\\
&\le & (\epsilon')^2\left( \frac{1+\epsilon}{1-\epsilon}\right)\|g_{x_*}(x_0)\|^2_{x_*} \quad\quad\quad\quad\quad\mbox{ (by \eqref{eq:tilde g})}.
\end{eqnarray*}
To complete the proof now only requires the triangle inequality,
\[
\|g_{x_*}(x_0)-\tilde d\|_{x_*} \le \|g_{x_*}(x_0)- d\|_{x_*} + \|d-\tilde d\|_{x_*},
\]
as well as the inequality from Theorem \ref{thm:approximate direction}.
\end{proof}

\section{Concluding remarks}
\label{sec:conclusion}
In this paper we have extended the SDP performance estimation analysis to second order methods, and
demonstrated an example of how to use the resulting error bounds in the complexity analysis of inexact interior point methods.
Our analysis of the interior point method of Abernethy and Hazan \cite{Abernethy_Hazan_2016} gives an outline of a new proof of
the polynomial complexity of the method.
Having said that, we have assumed in this paper that the gradient and Hessian of the entropic
barrier function may be computed within a fixed relative accuracy through sampling.
It is possible to complete the analysis of the sampling process in the spirit of the work of  Kalai and Vempala \cite{Kalai-Vempala 2006}
and Lov\'asz and Vempala \cite{Lovasz_Vempala_2007}. This requires detailed analysis
 of the hit-and-run sampling method for log-concave distributions on convex bodies, combined with the
 self-concordance property of the entropic barrier function.
Since the resulting analysis is lengthy, and of a very different type than that presented here, we
present it in a separate work, namely \cite{Badenbroek_DeKlerk2018}.

\subsection*{Acknowledgement}
Etienne de Klerk would like to thank Riley Badenbroek for pointing out the result in Theorem
\ref{lemma:ImproveDistanceZ}, and providing its proof, and for also pointing out a mistake in the proof
of Theorem \ref{thm:complexity bound IPM iterations} in an earlier version of this paper.

\bibliographystyle{siamplain}

\begin{thebibliography}{99}

\bibitem{Abernethy_Hazan_2016}
J.\ {Abernethy} and E.\ {Hazan}.
    Faster Convex Optimization: Simulated Annealing with an Efficient Universal Barrier.
    \newblock
 Proceedings of The 33rd International Conference on Machine Learning, PMLR 48:2520--2528, 2016.
 \url{http://proceedings.mlr.press/v48/abernethy16.html}


\bibitem{Azagra_Mudara_2017}
D.\ Azagra and C.\ Mudarra. An Extension Theorem for convex functions of class $C^{1,1}$ on Hilbert spaces.
 \emph{Journal of Mathematical Analysis and Applications}, 446.2, 1167--1182, 2017.


\bibitem{Adamczak_et_al_2010}
R. Adamczak, A.E. Litvak, A. Pajor,
and N. Tomczak-Jaegermann.
Quantitative estimates of the convergence
of the empirical covariance matrix
in log-concave ensembles.
\emph{Journal of the AMS },
 23(2),  535--561, 2010.

\bibitem{Badenbroek_DeKlerk2018}
R.~Badenbroek and E.~de Klerk.
Complexity analysis of a sampling-based interior point method for convex optimization.
	arXiv:1811.07677 [math.OC], 2018.

\bibitem{Bolte2020}
J.~Bolte and E.~Pauwels
\newblock Curiosities and counterexamples in smooth convex optimization.
\newblock arXiv:2001.07999 [math.OC], 2018.



\bibitem{Bubeck_Eldan_2014}
S.~Bubeck and R.~Eldan.
The entropic barrier: a simple and optimal universal self-concordant barrier.
\newblock In: \emph{Conference on Learning Theory}, 279--279, 2015.


\bibitem{Crouzeix_1977}
J.P. Crouzeix.
A relationship between the second derivatives
of a convex function and of its conjugate.
\emph{Mathematical Programming}, 13  364--365, 1977.

\bibitem{Cyrus_Hu_Scoy_Lessard_2018}
S.~Cyrus, B.~Hu, B.~Van Scoy, and L.~Lessard.
A Robust Accelerated Optimization Algorithm
for Strongly Convex Functions.
\emph{Proceedings of the 2018 Annual American Control Conference (ACC)}, pp. 1376--1381, 2018.


\bibitem{Daspremont2008}
A.~d'Aspremont.
\newblock Smooth optimization with approximate gradient.
\newblock \emph{SIAM Journal on Optimization}, 19(3), 1171--1183, 2008.

\bibitem{Devolder2014}
O. Devolder, F. Glineur, and Y. Nesterov.
\newblock First-order methods of smooth convex optimization with inexact oracle.
\newblock \emph{Mathematical Programming}, 146(1-2), 37--75, 2014.

\bibitem{drori2016}
Y. Drori and M. Teboulle.
\newblock An optimal variant of {K}elley's cutting-plane method.
\newblock {\em Mathematical Programming}, 160(1-2):321--351, 2016.

\bibitem{drori2018}
Y.~Drori. \newblock
On the Properties of Convex Functions over Open Sets.
\newblock
	arXiv:1812.02419 [math.OC], 2018.

\bibitem{Drori_Taylor}
Y.~Drori and A.B.~Taylor.
Efficient first-order methods for convex minimization: a
constructive approach.
\emph{Mathematical Programming}, available online:
\url{https://doi.org/10.1007/s10107-019-01410-2}


\bibitem{drori2014contributions}
Y.~Drori.
\newblock {\em Contributions to the Complexity Analysis of Optimization
	Algorithms}.
\newblock PhD thesis, Tel-Aviv University, 2014.

\bibitem{drori2014}
Y.~Drori and M.~Teboulle.
\newblock Performance of first-order methods for smooth convex minimization: a
  novel approach.
\newblock {\em Mathematical Programming}, 145(1-2):451--482, 2014.

\bibitem{gu2019a}
G.~Gu and J.~Yang.
\newblock Optimal nonergodic sublinear convergence rate of proximal point
algorithm for maximal monotone inclusion problems.
\newblock {\em arXiv:1904.05495 [Math.OC]}, 2019.

\bibitem{gu2019b}
G.~Gu and J.~Yang.
\newblock On the optimal ergodic sublinear convergence rate of the relaxed
proximal point algorithm for variational inequalities.
\newblock {\em arXiv:1905.06030 [Math.OC]}, 2019.

\bibitem{Kalai-Vempala 2006}
A.~T. Kalai and S.~Vempala.
 Simulated annealing for convex optimization.
 \emph{Mathematics of Operations Research}, 31(2),
  253--266, 2006.

\bibitem{kim2016}
D.~Kim and J.F.~Fessler.
\newblock Optimized first-order methods for smooth convex minimization.
\newblock {\em Mathematical Programming}, 159(1-2), 81--107, 2016.

\bibitem{kim2018optimizing}
D.~Kim and J.~A.~Fessler.
\newblock Optimizing the efficiency of first-order methods for decreasing the
gradient of smooth convex functions.
\newblock {\em arXiv:1803.06600 [Math.OC]}, 2018.

\bibitem{kim2019accelerated}
D.~Kim.
\newblock Accelerated proximal point method for maximally monotone operators.
\newblock {\em arXiv:1905.05149 [Math.OC]}, 2019.


\bibitem{Klerk_Glineur_Taylor_2016}
E. de Klerk,  F.~Glineur, and A.B. Taylor.
\newblock  On the worst-case complexity of the gradient method with exact line search for smooth strongly convex functions.
\newblock \emph{Optimization Letters}, 11(7), 1185--1199, 2017.

\bibitem{Lessard_Recht_Packard_2016}
L.~Lessard, B.~Recht, and A.~Packard.
\newblock Analysis and Design of Optimization Algorithms via Integral Quadratic Constraints.
\newblock\emph{SIAM Journal on Optimization}, 26(1), 57--95, 2016.

\bibitem{Li_et_al_2016}
J.\ Li, M.S.\ Andersen, and L.\ Vandenberghe. Inexact proximal Newton methods for self-concordant
functions.  \emph{Mathematical Methods of Operations Research}, 85, 19--41, 2017.

\bibitem{lieder2017convergence}
F.~Lieder.
\newblock On the convergence rate of the {H}alpern-iteration.
\newblock Technical report, 2017.


\bibitem{Lovasz_Vempala_2007}
L. Lovasz and S. Vempala. The geometry of logconcave functions and sampling algorithms. \emph{Random
Structures \& Algorithms}, 30(3):307--358, 2007.


\bibitem{nesterov_lectures_convex_opt}
Yu. Nesterov.
\newblock {\em Lectures on convex optimization}, 2nd ed.
\newblock Springer Optimization and Its Applications 137, Springer Nature Switzerland, 2018.

\bibitem{int:Nesterov5}
Yu.~Nesterov and A.S.~Nemirovski,
{\em Interior point polynomial algorithms in convex programming}.
SIAM, 1994.

\bibitem{Polyak1971}
B.T. Polyak.
\newblock Convergence of methods of feasible directions in extremal problems.
\newblock USSR Computational Mathematics and Mathematical Physics, 11(4), 53--70, 1971.

\bibitem{Renegar2001a}
J.~Renegar,
{\em A Mathematical View of Interior-Point Methods in Convex
  Optimization},
SIAM, 2001.


\bibitem{Ryu2018}
E.~K.~Ryu, A.~B. Taylor, C.~Bergeling, and P.~Giselsson.
\newblock Operator splitting performance estimation: Tight contraction factors
and optimal parameter selection.
\newblock {\em arXiv:1812.00146 [Math.OC]}, 2018.

\bibitem{Smith 1984}
 R. Smith. Efficient Monte Carlo procedures for generating points uniformly
distributed over bounded regions, \emph{Operations Research} 32(6),
1296--1308, 1984.

\bibitem{Schmidt2011}
M. Schmidt, N. {Le Roux}, and F. Bach. Convergence rates of inexact proximal-gradient methods for convex optimization. In \emph{Advances in neural information processing systems}, 1458--1466, 2011.

\bibitem{Taylor_Glineur_Hendrickx_2018}
A.B. Taylor, J.M. Hendrickx, and F.~Glineur.
\newblock Exact worst-case convergence rates of the proximal gradient method for composite convex minimization.
\newblock \emph{Journal of Optimization Theory and Applications}, 178(2), 455--476, 2018.

\bibitem{Taylor_Glineur_Hendrickx_2017a}
A.B. Taylor, J.M. Hendrickx, and F.~Glineur.
\newblock Smooth strongly convex interpolation and
exact worst-case performance of first-order methods. \newblock \emph{Mathematical Programming}, 161(1-2):307--345, 2017.

\bibitem{Taylor_Glineur_Hendrickx_2017b}
A.B. Taylor, J.M. Hendrickx, and F.~Glineur.
\newblock Exact worst-case performance of first-order methods for composite convex optimization.
\newblock \emph{SIAM Journal on Optimization}, 27(3), 1283--1313, 2017.

\bibitem{Taylor_Scoy_Lessard}
A.~Taylor, B.~Van Scoy, and L.~Lessard
\newblock Lyapunov functions for first-order methods: Tight automated convergence guarantees.
\newblock In Proceedings of the 35th International Conference on Machine Learning (ICML), pages 4897--4906, 2018.

\bibitem{van2018fastest}
B.~Van~Scoy, R.~A.~Freeman, and K.~~M. Lynch.
\newblock The fastest known globally convergent first-order method for
minimizing strongly convex functions.
\newblock {\em IEEE Control Systems Letters}, 2(1):49--54, 2018.
\end{thebibliography}

\appendix
\section{Worst-case example for Theorem \ref{thm:ELS improvements}}
\label{appendix:B}
Here,  show that the bound
\[\norm{g(x_1)}\leq \left(\varepsilon+\sqrt{1-\varepsilon^2}\tfrac{(1-\kappa)}{2\sqrt{\kappa}}\right)\norm{g(x_0)}\]
from Theorem \ref{thm:ELS improvements} cannot be improved on the class of smooth strongly convex functions, by giving an example of
a function $f\in\mathcal{F}_{\mu,L}(\mathbb{R}^n)$ where the bound holds with equality.

To this end, consider the two triplets $(x_0,g_0,f_0),(x_1,g_1,f_1)\in\mathbb{R}^2\times\mathbb{R}^2\times\mathbb{R}$:
\begin{equation*}
\begin{aligned}
x_0=\begin{bmatrix}
0\\0
\end{bmatrix}, \quad g_0=\begin{bmatrix}
1\\0
\end{bmatrix},\quad
f_0=0,
\end{aligned}
\end{equation*}
and
\begin{equation*}
\begin{aligned}
x_1=\begin{bmatrix}
-\tfrac{\left(1-\epsilon ^2\right) \left(1+\kappa\right)}{2 \mu }\\\tfrac{\epsilon  \sqrt{1-\epsilon ^2} \left(1+\kappa\right)}{2 \mu }
\end{bmatrix},\quad g_1=\begin{bmatrix}
\tfrac{\sqrt{1-\epsilon ^2} \epsilon  \left(1-\kappa\right)}{2 \sqrt{\kappa}}+\epsilon ^2\\
\frac{\left(1-\epsilon ^2\right) \left(1-\kappa\right)}{2 \sqrt{\kappa}}+\epsilon  \sqrt{1-\epsilon ^2}
\end{bmatrix},\quad
f_1=-\tfrac{\left(1-\epsilon ^2\right) \left(1+\kappa\right)}{4 \mu },
\end{aligned}
\end{equation*}
along with the inexact search direction $d\in\mathbb{R}^2$
\begin{equation*}
d=\begin{bmatrix}
1-\epsilon ^2\\-\epsilon  \sqrt{1-\epsilon ^2}
\end{bmatrix},
\end{equation*}
where $\kappa:=\tfrac{\mu}{L}$ is the (inverse) condition ratio, and $\epsilon \ge 0$. The following facts hold:
\begin{enumerate}
	\item the pair of triplets satisfies the conditions
	\begin{equation*}
	\begin{aligned} f_0-f_1+\inner{g_0}{x_1-x_0}+\tfrac{1}{2(1-\mu/L)}\left(\tfrac1L\normsq{g_0-g_1}+\mu\normsq{x_0-x_1}-2\tfrac{\mu}{L}\inner{g_0-g_1}{x_0-x_1}\right)=0,\\
	f_1-f_0+\inner{g_1}{x_0-x_1}+\tfrac{1}{2(1-\mu/L)}\left(\tfrac1L\normsq{g_0-g_1}+\mu\normsq{x_0-x_1}-2\tfrac{\mu}{L}\inner{g_0-g_1}{x_0-x_1}\right)=0,
	\end{aligned}
	\end{equation*}
	\item $x_1=x_0-\gamma d$ with $\gamma=\tfrac{1+\kappa}{2 \mu }$,
	\item $\inner{d}{g_1}=0$,
	\item $\norm{g_0-d}=\varepsilon\norm{g_0}$,
	\item $\tfrac{\norm{g_1}}{\norm{g_0}}=\left(\varepsilon+\sqrt{1-\varepsilon^2}\tfrac{(1-\kappa)}{2\sqrt{\kappa}}\right)$.
\end{enumerate}
Therefore, by \cite[Theorem 4]{Taylor_Glineur_Hendrickx_2017a}, there exists a function $f\in\mathcal{F}_{\mu,L}(\mathbb{R}^n)$ satisfying
\[f(x_0)=f_0,\, f(x_1)=f_1,\, g(x_0)=g_0,\, g(x_1)=g_1.\]
For this function, one iteration of gradient descent with exact line search starting at $x_0$ achieves exactly
\[\norm{g(x_1)}=\left(\varepsilon+\sqrt{1-\varepsilon^2}\tfrac{(1-\kappa)}{2\sqrt{\kappa}}\right)\norm{g(x_0)},\]
yielding the desired result.

\section{Proof of Theorem \ref{thm:FS improvements}}
\label{Appendix:A}
\subsubsection*{{Convergence of gradient norm}}
As in the proof of Theorem \ref{thm:ELS improvements}, consider the following inequalities along with their associated multipliers:
{\begin{align*}
& {\langle{g_0-g_1},{x_0-x_1}\rangle\geq \frac{1}{1-\frac{\mu}{L}}\left(\frac{1}{L}\|{g_0-g_1}\|^2+\mu\|{x_1-x_0}\|^2-2\frac{\mu}{L}\langle{g_0-g_1},{x_0-x_1}\rangle\right)}
&&:\lambda_0,\\
&\normsq{d-g_0}-\varepsilon^2\normsq{g_0}\leq 0 &&:{\lambda_1}.
\end{align*}}
In the following developments, we will also use the form of the {algorithm:
\[ x_1=x_0-\gamma d,\] and the notation $\rho_{\varepsilon}(\gamma):=1-(1-\varepsilon)\mu\gamma$.} Recall that we want to prove that the rate $\rho^2_{\varepsilon}(\gamma)$ is valid on the interval
\[\gamma\in \left[0, \frac{2\mu-\varepsilon(L+\mu)}{(1-\varepsilon)\mu(L+\mu)}\right],\]
when $\frac{2\mu-\varepsilon(L+\mu)}{(1-\varepsilon)\mu(L+\mu)}\geq 0\Leftrightarrow \varepsilon\leq \frac{2\mu}{L+\mu}$ (that is, we only consider $\gamma\geq 0$).
We use the following values for the multipliers:
\begin{align*}
& \lambda_{0}=\frac{2}{\gamma(1-\varepsilon)}\rho_{\varepsilon}(\gamma),\\
& \lambda_1=\frac{\gamma\mu}{\varepsilon}\rho_{\varepsilon}(\gamma).
\end{align*}
In that case, one can write the weighted sum of the previous constraints in the following form:
\begin{eqnarray*}
\rho^2_{\varepsilon}(\gamma)\normsq{g_0}&\geq& \normsq{g_1}+\frac{2-(1-\varepsilon) \gamma (L+\mu)}{(1-\varepsilon) \gamma (L-\mu)} \normsq{\frac{\gamma (L+\mu) ((\varepsilon-1) \gamma \mu+1)}{(\varepsilon-1) \gamma (L+\mu)+2}d -\frac{2  ((\varepsilon-1) \gamma \mu+1)}{(1-\varepsilon) \gamma (L+\mu)+2}g_0+g_1}\\
&+&\frac{(1-\varepsilon) \gamma (1-(1-\varepsilon) \gamma \mu) \Big(-(1-\varepsilon) \gamma \mu (L+\mu)-\varepsilon (L+\mu)+2 \mu\Big)}{\varepsilon (2-(1-\varepsilon) \gamma (L+\mu))}\normsq{\frac{1}{\varepsilon-1}d+g_0}.
\end{eqnarray*}
Therefore, the guarantee
\[ \rho^2_{\varepsilon}(\gamma)\normsq{g_0}\geq \normsq{g_1}\]
is valid as long as both the Lagrange multipliers, and the coefficients of the norms in the previous expression are nonnegative. That is, under the following conditions:
\begin{itemize}
\item[-] the Lagrange multipliers are nonnegative as long as $\rho_{\varepsilon}(\gamma)\geq 0$, that is, when
\[ \gamma \leq \frac{1}{(1-\varepsilon)\mu},\]
{which is valid for all values of $\gamma$ in the interval of interest (see below).}
\item[-] The coefficients of the norms are also nonnegative, since
\begin{align*}
2-(1-\varepsilon) \gamma (L+\mu)&\geq 0 \Leftrightarrow \gamma\leq\frac{2}{(1-\varepsilon)(L+\mu)},\\
(1-(1-\varepsilon) \gamma \mu)&\geq 0\Leftrightarrow \gamma\leq \frac{1}{(1-\varepsilon)\mu},\\
\Big(-(1-\varepsilon) \gamma \mu (L+\mu)-\varepsilon (L+\mu)+2 \mu\Big)&\geq 0 \Leftrightarrow \gamma\leq \frac{2\mu-\varepsilon(L+\mu)}{(1-\varepsilon)\mu(L+\mu)},
\end{align*}
{which are all valid on the interval of interest for $\gamma$, as}:
\[\frac{2\mu-\varepsilon(L+\mu)}{(1-\varepsilon)\mu(L+\mu)}=\frac{2}{(1-\varepsilon)(L+\mu)}-\frac{\varepsilon}{(1-\varepsilon)\mu}\leq \frac{2}{(1-\varepsilon)(L+\mu)}\leq \frac{1}{(1-\varepsilon)\mu}.\]
\end{itemize}

\subsubsection*{{Convergence of distance to optimality}}
{Consider the following inequalities and the associated multipliers:}
\begin{align*}
&{\frac{1}{1-\frac{\mu}{L}}\left(\frac{1}{L}\|{g_0}\|^2+\mu\|{x_0-x_*}\|^2-2\frac{\mu}{L}\langle{g_0},{x_0-x_*}\rangle\right)+\langle{g_0},{x_*-x_0}\rangle\leq 0}\quad &&:{\lambda_0},\\
&\normsq{d-g_0}-\varepsilon^2\normsq{g_0}\leq 0 &&:{\lambda_1}.
\end{align*}
As in the case of the gradient norm, we use the notation
$\rho_{\varepsilon}(\gamma):=1-(1-\varepsilon)\mu\gamma$. Let us recall that we want to prove that the rate $\rho^2_{\varepsilon}(\gamma)$ is valid on the interval
\[\gamma\in \left[0, \frac{2\mu-\varepsilon(L+\mu)}{(1-\varepsilon)\mu(L+\mu)}\right],\]
when $\frac{2\mu-\varepsilon(L+\mu)}{(1-\varepsilon)\mu(L+\mu)}\geq 0\Leftrightarrow \varepsilon\leq \frac{2\mu}{L+\mu}$ (we only consider $\gamma\geq 0$).
We now use the following values for the multipliers:{
\begin{align*}
& \lambda_{0}=2\gamma(1-\varepsilon)\rho_{\varepsilon}(\gamma),\\
& \lambda_1=\frac{\gamma}{\mu\varepsilon}\rho_{\varepsilon}(\gamma).
\end{align*}
In} that case, one can write the weighted sum of the previous constraints in the following form:
\begin{align*}
&(1-\gamma\mu(1-\varepsilon))^2\norm{x_0-x_*} \geq \norm{x_1-x_*}+\gamma \mu^2(1-\varepsilon)\frac{2-\gamma (1-\varepsilon)(L+\mu)}{L-\mu} \times \\
&\normsq{\frac{L-\mu}{(1-\varepsilon) \mu^2 (2-\gamma(1-\varepsilon)
  (L+\mu))}d-\frac{ (L+\mu) (1-\gamma\mu(1-\varepsilon))}{\mu^2 (2-\gamma(1-\varepsilon) (L+\mu))}g_0+x_0-x_*} +\\
  &\gamma\frac{ (1-\gamma\mu (1-\varepsilon))
   (2\mu-\varepsilon(L+\mu)-\gamma\mu(1-\varepsilon)(L+\mu))}{\varepsilon \mu^2(1-\varepsilon)  (2-\gamma(1-\varepsilon)
    (L+\mu))}\normsq{d-(1-\varepsilon) g_0}.
\end{align*}
Hence, all coefficients and multipliers are positive as long as
\begin{align*}
&2\mu-\varepsilon(L+\mu)-\gamma\mu(1-\varepsilon)(L+\mu)\geq 0\Leftrightarrow \gamma\leq \frac{\mu-\varepsilon(L+\mu)}{(1-\varepsilon)\mu(L+\mu)},\\
&1-(1-\varepsilon)\mu\gamma\geq 0\Leftrightarrow \gamma\leq \frac{1}{(1-\varepsilon)\mu},\\
&2-\gamma(1-\varepsilon)(L+\mu)\geq 0 \Leftrightarrow \gamma\leq \frac{2}{(1-\varepsilon)(L+\mu)}.
\end{align*}
We refer to previous discussions for the details leading to the conclusion:
\[(1-\gamma\mu(1-\varepsilon))^2\norm{x_0-x_*} \geq \norm{x_1-x_*}.\]

\subsubsection*{Convergence of function values}
{As in the previous section, we use the notation $\rho_{\varepsilon}(\gamma):=1-(1-\varepsilon)\mu\gamma$, and consider the case \[\gamma\in \left[0, \frac{2\mu-\varepsilon(L+\mu)}{(1-\varepsilon)\mu(L+\mu)}\right].\]
For proving the desired convergence rate in terms of function values, we consider the following set of inequalities (and associated multipliers):}
\begin{align*}
&{\begin{aligned}
f_0 - f_1 &-\langle{{g}_1},{{x}_0-{x}_1}\rangle\\&\geq \frac{1}{2(1-\mu/L)}\left( \frac{1}{L}\norm{{g}_0-{g}_1}^2+ \mu \norm{{x}_0-{x}_1}^2 - 2\frac{\mu}{L} \langle{{g}_1-{g}_0},{{x}_1-{x}_0}\rangle\right)\end{aligned}} &&:\lambda_{01}=\rho_{\varepsilon}(\gamma), \\
&{\begin{aligned}
f_* - f_0 &-\langle{{g}_0},{{x}_*-{x}_0}\rangle\\&\geq \frac{1}{2(1-\mu/L)}\left( \frac{1}{L}\norm{{g}_*-{g}_0}^2+ \mu \norm{{x}_*-{x}_0}^2 - 2\frac{\mu}{L} \langle{{g}_0-{g}_*},{{x}_0-{x}_*}\rangle\right)\end{aligned}}
&&:\lambda_{*0}=\rho_{\varepsilon}(\gamma)(1-\rho_{\varepsilon}(\gamma)),\\
&{\begin{aligned}
f_* - f_1 &-\langle{{g}_1},{{x}_*-{x}_1}\rangle\\&\geq \frac{1}{2(1-\mu/L)}\left( \frac{1}{L}\norm{{g}_*-{g}_1}^2+ \mu \norm{{x}_*-{x}_1}^2 - 2\frac{\mu}{L} \langle{{g}_1-{g}_*},{{x}_1-{x}_*}\rangle\right)\end{aligned}}
&&:\lambda_{*1}=1-\rho_{\varepsilon}(\gamma),\\
&\normsq{d-g_0}-\varepsilon^2\normsq{g_0}\leq 0 &&:{\lambda_2=\frac{\gamma}{2\varepsilon} \rho_{\varepsilon}(\gamma)}.
\end{align*}{Note that $\rho_{\varepsilon}(\gamma)\leq 1$ and that the multipliers are nonnegative in the cases of interest}. We can write the weighted sum of the previous constraints in the following form :
\begin{align*}
&\rho_{\varepsilon}^2(\gamma) (f(x_0)-f(x_*))\\
&\geq f(x_1)-f(x_*)\\&+\frac{\gamma \rho_{\varepsilon}(\gamma) (L (-2 \varepsilon \gamma \mu+\rho_{\varepsilon}(\gamma)-1)+\mu (\rho_{\varepsilon}(\gamma)+1))}{2 \varepsilon (L (\rho_{\varepsilon}(\gamma)-1)+\mu (\rho_{\varepsilon}(\gamma)+1))}\normsq{d+\frac{g_0 ((\varepsilon+1) L (\rho_{\varepsilon}(\gamma)-1)-(\varepsilon-1) \mu (\rho_{\varepsilon}(\gamma)+1))}{L (2 \varepsilon \gamma \mu-\rho_{\varepsilon}(\gamma)+1)-\mu (\rho_{\varepsilon}(\gamma)+1)}}\\
&+\frac{L \mu \left(1-\rho_{\varepsilon}^2(\gamma)\right)}{2 (L-\mu)}\normsq{-\frac{d \gamma}{\rho_{\varepsilon}(\gamma)+1}-\frac{g_0 \rho_{\varepsilon}(\gamma)}{\mu \rho_{\varepsilon}(\gamma)+\mu}-\frac{g_1}{\mu \rho_{\varepsilon}(\gamma)+\mu}+x_0-x_*}\\
&+\frac{\rho_{\varepsilon}(\gamma)(L+\mu)-(L-\mu)}{2 \mu (\rho_{\varepsilon}(\gamma)+1) (L-\mu)}\normsq{\frac{2 d \gamma L \mu \rho_{\varepsilon}(\gamma)}{L (\rho_{\varepsilon}(\gamma)-1)+\mu (\rho_{\varepsilon}(\gamma)+1)}+\frac{g_0 \rho_{\varepsilon}(\gamma) (L (\rho_{\varepsilon}(\gamma)-1)-\mu (\rho_{\varepsilon}(\gamma)+1))}{L (\rho_{\varepsilon}(\gamma)-1)+\mu (\rho_{\varepsilon}(\gamma)+1)}+g_1}\\
&-\frac{\rho_{\varepsilon}(\gamma) ((\varepsilon+1) \gamma L-\rho_{\varepsilon}(\gamma)-1) ((\varepsilon-1) \gamma \mu-\rho_{\varepsilon}(\gamma)+1)}{L (2 \varepsilon \gamma \mu-\rho_{\varepsilon}(\gamma)+1)-\mu (\rho_{\varepsilon}(\gamma)+1)}\normsq{g_0}\\
&\geq f(x_1)-f(x_*).
\end{align*}

{In order to prove the last inequality, we have to show that the coefficients of the norms of the decomposition are nonnegative.
\begin{enumerate}
\item Term 1: substituting $\rho_{\varepsilon}(\gamma)$ by its expression, nonnegativity of the coefficient follows from
\begin{align*}
&\gamma\leq \frac{1}{(1-\varepsilon)\mu}\\
&\gamma\leq \frac{2 - \varepsilon(L-\mu) \left(L\mu (1-\varepsilon^2)\right)^{-1/2}}{\left(L + \mu\right)}\\
&\gamma\leq \frac{2}{(1-\varepsilon)(L+\mu)}
\end{align*}
which hold, as $\gamma\leq \frac{2 - \varepsilon(L-\mu) \left(L\mu (1-\varepsilon^2)\right)^{-1/2}}{\left(L + \mu\right)}\leq\frac{2}{(1-\varepsilon)(L+\mu)}\leq\frac{1}{(1-\varepsilon)\mu}$ on the interval of interest for $\gamma$.
\item Term 2: always nonnegative as $0\leq \rho_{\varepsilon}(\gamma)\leq 1$ on the interval of interest for $\gamma$.
\item Term 3: substituting $\rho_{\varepsilon}(\gamma)$ by its expression, one can easily verify that the coefficient is positive when \[ \gamma \leq \frac{2}{(1-\varepsilon)(L+\mu)},\]
which is true on the interval of interest for $\gamma$.
\item Term 4: cancels out by substituting $\rho_{\varepsilon}(\gamma)$ by its expression.
\end{enumerate}}

 \end{document}